\documentclass[10pt]{article}

\usepackage[a4paper, margin=2.5cm]{geometry}

\usepackage[figuresright]{rotating}
\usepackage{multirow} 
\usepackage{amsfonts,amsmath,latexsym,amssymb}
\usepackage{amsthm}
\usepackage{mathrsfs,upref}
\usepackage{url}
\usepackage{anyfontsize}
\usepackage{hyperref}

\DeclareMathAlphabet{\mathbbold}{U}{bbold}{m}{n}
\def\1{\mathbbold{1}}
\def\0{\mathbbold{0}}

\def\RR{{\mathbb{R}}}
\def\NN{{\mathbb{N}}}
\def\N{\mathcal{N}}
\def\M{\mathcal{M}}

\newcommand{\diag}{\mbox{\textrm diag$\,$}}
\newcommand{\mc}{\mathcal}

\newtheorem{theorem}{Theorem}[section]       
\newtheorem{lemma}[theorem]{Lemma}               
\newtheorem{corollary}[theorem]{Corollary}

\newtheorem{definition}[theorem]{Definition}
\newtheorem{reduction}[theorem]{Reduction}
\newtheorem{example}[theorem]{Example}

\newtheorem{remark}[theorem]{Remark}

\begin{document}

		\title{Majorizations for probability distributions, column stochastic matrices and their linear preservers}

        \author{Pavel Shteyner \footnote{Department of Mathematics, Bar-Ilan University, Ramat-Gan, 5290002, Israel.	{\tt pavel.shteyner@biu.ac.il}}\\
	}

             \maketitle
		\begin{abstract}
		In this paper, we study majorization for probability distributions and column stochastic matrices. We show that majorizations in general can be reduced to the aforementioned sets. We characterize linear operators that preserve majorization for probability distributions, and show their equivalence to operators preserving vector majorization. Our main result provides a complete characterization of linear preservers of strong majorization for column stochastic matrices, revealing a richer structure of preservers than in the standard setting. As a prerequisite to this characterization, we solve the problem of characterizing linear preservers of majorization for zero-sum vectors, which yields a new structural insight into the classical results of Ando and of Li and Poon.
        
		\end{abstract}

	

\section{Introduction and main results}

Let $M_{n,m}$ denote the space of all real $n \times m$ matrices, with $M_n$ used when $m=n$. For a matrix $X$, its $(i,j)$-entry is denoted by $x_{ij}$. Throughout this paper, we use the following standard matrices:
\begin{itemize}
    \item $I$ denotes the identity matrix
    \item $J$ denotes the matrix of all ones
    \item $0$ denotes the zero vector and the zero matrix of the appropriate sizes.
\end{itemize}

For a matrix $A$, we denote its $j$-th column by $A^{(j)}$ and its $i$-th row by $A_{(i)}$. A matrix represented by its columns is written as $\left(\begin{array}{c|c|c|c} A^{(1)} \ & \ A^{(2)} \ & \ \ldots \ & \ A^{(m)} \end{array}\right)$. The $n \times 2$ submatrix consisting of columns $j_1$ and $j_2$ of $A$ is denoted by $A^{(j_1, j_2)}$.

We write $A \geq 0$ (resp. $v \geq 0$) if every entry of the matrix $A$ (resp. the vector $v$) is nonnegative.

Vectors in $\mathbb{R}^n$ are treated as column vectors and identified with the corresponding $n$-tuples. We denote:
\begin{itemize}
    \item $e_j$ as the $j$-th standard basis vector
    \item $e = (1 \ \ldots \ 1)^t$ as the vector of all ones
    \item $\max(v)$ and $\min(v)$ as the maximum and minimum entries of $v \in \mathbb{R}^n$
    \item $\N = \{1, \ldots, n\}$ and $\M = \{1, \ldots, m\}$
    \item For $v \in \mathbb{R}^n$ and $\mathcal{I} \subseteq \N$, $v_\mathcal{I}$ denotes the vector comprising the entries of $v$ indexed by $\mathcal{I}$
    \item $v^+$ denotes the sum of all positive entries of $v \in \mathbb{R}^n$.
\end{itemize}

We use $M_{n,m}(0,1)$ and $\{0,1\}^n$ to denote the set of $(0,1)$-matrices of size $n \times m$ and $(0,1)$-vectors of size $n$, respectively.

Key vector sets of interest include:
\begin{itemize}
    \item $\0^n = \{v \in \mathbb{R}^n : e^t v = 0\}$, the set of all zero-sum vectors
    \item $\1^n = \{v \in \mathbb{R}^n : e^t v = 1 \text{ and } v \geq 0\}$, the set of probability distributions, see Definition \ref{def:prob}.
\end{itemize}

The set of $n \times n$ permutation matrices is denoted by $P(n)$, with $P_{(ij)} \in P(n)$ representing the matrix obtained by interchanging rows $i$ and $j$ of the identity matrix $I$. For a linear operator $\Phi$ on $\mathbb{R}^n$, its matrix representation in the standard basis is denoted by $[\Phi]$. The cardinality of a set $X$ is denoted by $|X|$.

\bigskip

Majorization is a powerful mathematical framework for comparing degrees of disorder in vectors and matrices, with applications ranging from linear algebra and economic inequality measures to quantum information theory and statistical experiments.
    
	The main notion of this theory is {\em vector majorization}. For a vector $x \in \RR^n$ we let $x^\downarrow$ denote the permutation of its entries in the non-increasing order. For $a,b \in \RR^n$ we say that	$a$ is {\em majorized} by $b$, denoted
	by $a \preceq b$,  if
	$$\sum_{j=1}^k a^\downarrow_{j} \le \sum_{j=1}^k b^\downarrow_{j} \text{ for } k=1, 2, \ldots, n-1 \text{ and } \sum_{j=1}^n a^\downarrow_{j} = \sum_{j=1}^n b^\downarrow_{j}.$$

A matrix is {\em row stochastic} (resp. {\em column stochastic}) if it
is nonnegative and each row (resp. column) sum equals to~$1$. A matrix is {\em doubly stochastic} if it is both row and column stochastic. The sets of all $n \times n$ row, column and doubly stochastic matrices are denoted by $\Omega_n^{row}, \Omega^{col}_n$ and $\Omega_n$, respectively. Similarly, the set of all $n\times m$ column stochastic matrices is denoted by $\Omega^{col}_{n, m}$.

The classical Birkhoff–von Neumann Theorem shows that $\Omega_n$ is the convex hull of $P(n)$.

\begin{theorem}\cite[Theorem I.2.A.2]{MaOlAr11}\label{Birkhoff}
	The elements of the set of permutation matrices $P(n)$ are the extreme points of $\Omega_n$.
	Moreover, $\Omega_n$ is the convex hull of matrices in $P(n)$.
\end{theorem}

The famous Hardy, Littlewood and P\'olya theorem expresses vector majorization in terms of doubly stochastic matrices.

\begin{theorem}\label{HLP}\cite[Theorem I.2.B.2]{MaOlAr11}
		Let $a, b \in \RR^n$. Then $a \preceq b$ if and only if $a = Db$ for some $D \in \Omega_n$.
\end{theorem}

This approach inspired many important generalizations of majorization to matrices, most notably, strong majorization.

\begin{definition} \label{def:strong}
\rm Let $A, B \in M_{n, m}$. {\em Strong majorization} is defined by $A \preceq^{s} B$ if $A=DB$ for some $D \in \Omega_n$.
\end{definition}

Thus vector majorization is nothing but strong majorization for single-column matrices.

In this paper, we investigate majorization for column stochastic matrices. In particular, we study vector majorization for probability distributions.

\begin{definition}\label{def:prob}\rm
    A vector $v \in \RR^n$ is a {\em probability distribution} when $v \geq 0$ and $e^t v = 1$. That is, a probability distribution is an $n\times 1$ column stochastic matrix. \end{definition}

Our investigations are motivated by several aspects. It turns out that majorization in general can be reduced to majorization for column stochastic matrices and probability distributions in the vector case. On the other hand, majorizations of these objects are very important in applications, for instance, in Quantum Information Theory and Theory of Statistical Experiments.

Vector majorization is often used to compare vectors and probability distributions, offering insights into the concepts of order and disorder in various systems. The quantum mechanical analogue of a probability distribution is the density matrix. A density matrix $A$ is said to be more chaotic or less pure than the density matrix $B$ if the spectrum of $A$ is majorized by the spectrum of $B$. This notion is usually called Quantum Majorization or simply Matrix Majorization, depending on the context. Spectra of density matrices are exactly probability distributions.

Majorization serves as a fundamental mathematical tool for understanding quantum information processes, providing deep insights into quantum measurement, entanglement transformation, and quantum dynamics. The relation $a \preceq b$ between vectors captures the intuitive notion that $a$ is "more mixed" or "more disordered" than $b$, providing a mathematical framework more precise than entropic measures for quantifying quantum disorder. This can be demonstrated through the concept of Schur-convex functions, which are real-valued functions that preserve the majorization relation between vectors. Both Shannon entropy and von Neumann entropy are Schur-concave functions, meaning they increase as the disorder of a system increases. However, because they are Schur-concave, they do not offer any more information than what is already captured by majorization itself. There are many problems where entropy is insufficient. The tools of majorization have been successfully applied to such problems. An example of an approach combining entropies and majorization can be found in \cite{Gour}.

The wide applicability of majorization to quantum mechanics was first established through two seminal results: Horn's lemma and Uhlmann's theorem that connect vector majorization to unitary matrices. Recent applications have revealed majorization's power in characterizing post-measurement states, conditions under which one entangled state can be converted to another through local operations and classical communication (LOCC), possible probability distributions that can appear in ensemble decompositions of quantum states and so on. A more comprehensive treatment on applications of Majorization in Quantum Information can be found in \cite{Nielsen}. 

Majorization for column stochastic matrices arises in the Theory of Statistical Experiments. In the discrete settings, statistical experiments are expressed as column stochastic matrices. The matrix majorization relation introduced in \cite{Dahl99} turns out to be the criterion of one experiment being more informative than the other. Strong majorization that we focus on is a particular (stronger) case of this majorization. Further information can be found in \cite{Torgersen, Dahl99b, Dahl99,DGS1_MMC}. Note that in these references, compared to the present paper, mostly the transposed versions of the notions are used. In particular, the matrices are row stochastic.

\bigskip

We establish straightforward reductions of majorization to column stochastic matrices and probability distributions. After that we study linear operators preserving such majorizations. In the context of majorization, order-preserving functions are called Schur-convex functions, see \cite{Schur} and \cite[Chapter I.3]{MaOlAr11}. Generalizing this approach, Linear Preserver Problems investigate linear maps that preserve majorization. Such maps have been extensively studied for different types of majorization and different restrictions on the matrix (resp. vector) set, see \cite{Ando89, BeasleyLee, BeasleyLeeLee,  GS_tuples, GS2_convecters, GS3_preservers_01, HR, HR2, LiPoon2001, Shteyner_CCM, Shteyner_conv_01, Shteyner_column:01}. For a general survey on Linear Preserver Problems that date back to Frobenius, see \cite{Pierce, LiPierce}.

The first result in this theory that concerns majorization is a characterization of linear operators preserving vector majorization, obtained by Ando.

 \begin{definition}\rm
    A linear operator $\phi$ on $\RR^n$ {\em preserves vector majorization} if $a \preceq b$ implies $\phi(a) \preceq \phi(b)$ for any $a, b \in \RR^n$.
 \end{definition}
	
	\begin{theorem}\label{Ando}\cite[Corollary 2.7]{Ando89}
		Let $\Phi$ be a linear operator on $\RR^n$. Then the following conditions are equivalent:
		\begin{enumerate}
			\item $\Phi$ preserves vector majorization.
			\item One of the following holds:
			\begin{enumerate}
				\item\label{a} $\Phi (x) = (e^tx)s$ for some $s \in \RR^{n}$.
				\item\label{b} $\Phi (x) = \alpha P x + \beta J x$ for some $\alpha, \beta \in \RR$ and some $P\in P(n)$.
			\end{enumerate}
		\end{enumerate}
	\end{theorem}

We refer to operators of the form \ref{a} as Ando's operators of the first type and operators of the form \ref{b} as Ando's operators of the second type.

We can similarly define linear operators preserving majorization for probability distributions.

 \begin{definition}\rm
    A linear operator $\phi$ on $\RR^n$ {\em preserves majorization for probability distributions} if $a \preceq b$ implies $\phi(a) \preceq \phi(b)$ for any $a, b \in \1^n$.
 \end{definition}

 For the sake of brevity, we sometimes simply say that $\phi$ preserves majorization on $\1^n$. In the same way, we define linear preservers of majorization on $\0^n$, linear preservers of strong majorization and strong majorization for column stochastic matrices.

 \begin{definition}\rm
    A linear operator $\Phi$ on $M_{n, m}$ {\em preserves strong majorization (resp. preserves strong majorization on $\Omega^{col}_{n, m}$)} if $A \preceq^s B$ implies $\Phi(A) \preceq^s \Phi(B)$ for any $A, B \in M_{n, m}$ (resp. for any $A, B 
\in \Omega^{col}_{n, m}$).
 \end{definition}

Vector majorization can be reduced to majorization for probability distributions, as we show in Section \ref{subsec:Reductions:vectors}. In Theorem \ref{thm:vector} we prove that a linear map preserves vector majorization if and only if it preserves majorization for probability distributions. This is proved via $(0, 1)$-vectors and the respective linear preservers, see \cite{GS3_preservers_01}. It turns out, however, that the same approach fails in the case of matrices.

The main result of the paper is Theorem \ref{thm} that characterizes linear operators preserving strong majorization for column stochastic matrices. Despite the fact that strong majorization in general can also be reduced to column stochastic matrices, and despite the simplicity of the reduction, we obtain a richer structure of preservers.

Notably, this problem turns out to be more complicated than other similar matrix majorization preserver problems in the following sense. A matrix operator on $M_{n, m}$ can be decomposed in a standard way into $m^2$ linear operators on $\RR^n$, see Definition \ref{def:decomposition}. In the case of preservers of majorizations for real or for $(0, 1)$-matrices, the $m^2$ components of the operator leave invariant the respective majorization in the vector case. This approach already gives a strong necessary condition which goes a long way toward solving the characterization problem. 

In the case of column stochastic matrices, however, we cannot claim that the components of a preserver leave invariant majorization for $n\times 1$ column stochastic matrices, i.e. probability distributions. Indeed, as we see from Theorem \ref{thm} and Example \ref{ex:last}, this is not the case. However, we find that these operators must preserve majorization for zero-sum vectors. The concise characterization of such operators is found in Theorem \ref{thm:0}. This result provides better insight into the nature of Ando's classical operators, showing that they emerge as two particular cases of this natural and simple general structure.

The rest of the paper is organized as follows. Section \ref{sec:Basic} contains several additional notions of majorization and preliminary results. In Section \ref{sec:Reduction} we show how majorization can be reduced to column stochastic matrices.  Section \ref{sec:1} provides a characterization of linear preservers of majorization for probability distributions. Section \ref{sec:0} solves the prerequisite problem of characterizing linear preservers of majorization for zero-sum vectors. The characterization of linear preservers of strong majorization for column stochastic matrices is the main result of the paper that can be found in Section \ref{sec:Result}.

 \section{Preliminaries}\label{sec:Basic}
    
In this paper, we demonstrate that strong majorization can be reduced to the case of column stochastic matrices. Since the same approach works for directional majorization and, to a slightly lesser extent, for weak majorization, we also investigate these related notions.

\begin{definition} \label{def:dir:weak} \rm
Let $A, B \in M_{n, m}$. 

\begin{itemize}
    \item {\em Directional majorization} is defined by $A \preceq^{d} B$ if $Av \preceq Bv$ for any $v \in \RR^m$.
    \item {\em Weak majorization} is defined by $A \preceq^{w} B$ if $A=RB$ for some $R \in \Omega^{row}_n$.
\end{itemize}
\end{definition}

Strong majorization implies directional majorization, and directional majorization implies weak majorization. None of the reverse implications is true in general. A necessary condition for directional majorization (and thus also for strong majorization) is the equality of column sums.

\begin{lemma}
    Let $A, B \in M_{n, m}$. If $A \preceq^d B$, then $e^tA = e^tB$.
\end{lemma}

Multiplication by an invertible matrix on the right does not affect these majorizations.

\begin{lemma}\label{lem:right.inv}\cite[Corollary 2.13]{GS2_convecters}
	Let $A, B \in M_{n, m}$. Then for any invertible $Y \in M_m$ the following holds:
	\begin{enumerate}
		\item $A \preceq^w B$ if and only if $AY \preceq^w BY$.
		\item $A \preceq^d B$ if and only if $AY \preceq^d BY$.
		\item $A \preceq^s B$ if and only if $AY \preceq^s BY$.
	\end{enumerate}
\end{lemma}

As a direct consequence, we obtain the following property.

\begin{corollary}\label{cor:AD<BD}
    Let $A, B \in M_{n, m}$. Let $D \in M_m$ be a nonsingular diagonal matrix. Then $A \preceq^s B$ if and only if $AD \preceq^s BD$. The same property is true for directional and weak majorizations.
\end{corollary}

For our convenience, we introduce the following  equivalence relations induced by majorization.

	\begin{definition}
        \rm \
        
		\begin{itemize}
			\item For $a, b \in \RR^n$ we write  $a\sim b$ if $a \preceq b \preceq a$;
			
			\item  For $A, B \in M_{n, m}$ we write $A \sim^s B$ if $A \preceq^s B \preceq^s A$.
		\end{itemize}
	\end{definition}

    Both relations turn out to be nothing but equality up to a permutation.
	
	\begin{lemma}\label{lem:vector:symmetry}
		Let $a, b \in \RR^n$. Then $a \sim b$ if and only if $a=Pb$ for some $P \in P(n)$.
	\end{lemma}
	
	Lemma \ref{lem:vector:symmetry} is a particular case of the following more general statement.
	
	\begin{lemma}\cite[Theorem 3.24]{PeMaSi05}\label{lem:sim:strong}
		Let $A, B \in M_{n, m}$. Then $A \sim^s B$ if and only if
			there exists $P \in P(n)$ such that $A = PB$.
	\end{lemma}

For more information on strong, directional and weak majorization, see \cite{PeMaSi05, MaOlAr11, DGS1_MMC, DGS2_01, DGS3_0+-1} and the references therein.

Majorizations for $(0, 1)$-matrices and vectors were investigated in \cite{DGS2_01}. It turns out that strong majorization on $M_{n, m}(0, 1)$ is precisely the equivalence relation $\sim^s$.
	
	\begin{theorem}\label{thm:01}\cite[Theorem 3.5 and Corollary 3.6]{DGS2_01}
		Let $A, B\in M_{n, m}(0, 1)$. Then
		$A\preceq^{s} B$ if and only if $A\sim^sB$.
	\end{theorem}

For vector majorization, this takes the following form.

 \begin{corollary}\label{cor:01:vector}
     Let $a, b \in \{0, 1\}^n$. Then $a \preceq b$ if and only if $e^ta = e^t b$.
 \end{corollary}	

Finally, we provide the following characterization of linear operators preserving strong majorization. It was obtained by Li and Poon in \cite{LiPoon2001}, see also \cite{BeasleyLee, GS2_convecters}.
	
	\begin{theorem}\cite[Theorem 2]{LiPoon2001}\label{thm:LiPoon}
		Let $\Phi$ be a linear operator on $M_{n, m}$. The following conditions are equivalent:
		\begin{enumerate}
			\item $\Phi$ preserves strong majorization.
			\item One of the following holds:
			\begin{enumerate}
				\item There exist $S_1, \ldots, S_m \in M_{n, m}$ such that $\Phi(X)=\sum\limits^{m}_{j=1} (e^t X^{(j)})S_j$.
				\item There exist $R, S \in M_m$ and $P \in P(n)$ such that $\Phi(X) = PXR + JXS$.
			\end{enumerate}
		\end{enumerate}
	\end{theorem}

\section{Reduction to column stochastic matrices}\label{sec:Reduction}

In this section, we establish that when dealing with strong or directional majorization, we can always reduce the problem to column stochastic matrices. The same reduction applies to vector majorization. When dealing with weak majorization, we can make either matrix in the majorization relation column stochastic. Additionally, we specifically examine the role of $(0, 1)$-matrices.

\subsection{Strong majorization}

We begin by establishing several key properties that enable reductions to column stochastic matrices. 

\begin{lemma}\label{lem:lambda}
    Let $A, B \in M_{n, m}$ and $\lambda \in \RR$, $\lambda \neq 0$. Then  $A \preceq^s B$ if and only if $\lambda A \preceq^s \lambda B$.
\end{lemma}

\begin{lemma}\label{lem:ev^t}
    Let $A, B \in M_{n, m}$ and $v \in \RR^m$. Then $A \preceq^s B$ if and only if $A + ev^t \preceq^s B + ev^t$.
\end{lemma}
\begin{proof}

    Let $A \preceq^s B$. Then $A = QB$ for some $Q \in \Omega_n$. Observe that $QB + Qe v^t = A + e v^t$. Therefore $A + e v^t \preceq^s B + e v^t$.

    Now assume that $A + e v^t \preceq^s B + e v^t$. By the proven above, $A = (A + e v^t) - e v^t \preceq^s (B + e v^t) - e v^t = B$.
\end{proof}

\begin{corollary}\label{cor:+J}
    Let $A, B \in M_{n, m}$ and $\lambda \in \RR$. Then  $A \preceq^s B$ if and only if $A + \lambda J \preceq^s B + \lambda J$.
\end{corollary}

We now show how the general case of strong majorization can be reduced to column stochastic matrices.

\begin{reduction}\label{def:A':v}
Given $A, B \in M_{n, m}$ we construct column stochastic matrices $A', B' \in M_{n, m}$ in the following way:
\begin{enumerate}
    \item Choose $\lambda \geq 0$ large enough so that the matrices $A_1 = A + \lambda J$ and $B_1 = B + \lambda J$ are nonnegative. By Corollary \ref{cor:+J}, $A \preceq^s B$ if and only if $A_1 \preceq^s B_1$.
    
    Note that $e^tA_1 = e^tA + n\lambda e^t$ and $e^tB_1 = e^tB + n\lambda e^t$.

    \item Choose $\mu > 0$ large enough so that the column sums of $A_2 = \frac{1}{\mu} A_1$ and $B_2 = \frac{1}{\mu} B_1$ do not exceed $1$. Due to Lemma \ref{lem:lambda}, we have $A_1 \preceq^s B_1$ if and only if $A_2 \preceq^s B_2$.
    
    Note that $e^t A_2 = \frac{1}{\mu} e^t A_1 = \frac{1}{\mu}(e^tA + n\lambda e^t)$ and $e^t B_2 = \frac{1}{\mu}(e^tB + n\lambda e^t)$.

    \item\label{item:A'} Define $v \in \RR^m$ by $v^t = \frac{1}{n}(e^t - e^t B_2)$ and let $A' = A_2 + ev^t$ and $B' = B_2 + ev^t$. Note that $v$ is nonnegative by the definition of $B_2$. Then $A'$ and $B'$ are also nonnegative.

    By Lemma \ref{lem:ev^t} we obtain that $A_2 \preceq^s B_2$ if and only if $A' \preceq^s B'$. Therefore $A \preceq^s B$ if and only if $A' \preceq^s B'$.
    
    Observe that $$e^t A'= e^tA_2 + e^tev^t = e^tA_2 + nv^t =  \frac{1}{\mu}(e^tA + n\lambda e^t) + e^t - \frac{1}{\mu}(e^tB + n\lambda e^t) = e^t + \frac{1}{\mu}(e^tA - e^tB)$$
    and $$e^t B' = e^tB_2 + e^tev^t = e^tB_2 + nv^t = e^t.$$

    Therefore $B'$ is column stochastic. Moreover, $e^tA = e^tB$ if and only if $e^tA' = e^tB'$.
\end{enumerate}
\end{reduction}

We have obtained

\begin{corollary}\label{cor:strong}
    Let $A, B \in M_{n, m}$. Let $A', B' \in M_{n, m}$ be obtained via Reduction \ref{def:A':v}. Then $A \preceq^s B$ if and only if $A' \preceq^s B'$. In addition, $B'$ is column stochastic.
    
    Since $e^t A = e^t B$ is a necessary condition for $A \preceq^s B$, we obtain that if the majorization holds, then $e^tA' = e^t$ and both $A', B'$ are column stochastic matrices.
\end{corollary}

To illustrate the reduction process, consider the following example.

\begin{example}
    Let us perform Reduction \ref{def:A':v} on $A = \left(\begin{matrix}
-1 & -2 & 4 & -6\\
1 & -4 & 2 & -6
\end{matrix}\right)$ and $B = \left(\begin{matrix}
3 & -6 & 0 & -6 \\
-3 & 0 & 6 & -6
\end{matrix}\right)$.

First, taking $\lambda = 6$ we obtain $A_1 = \left(\begin{matrix}
5 & 4 & 10 & 0\\
7 & 2 & 8 & 0
\end{matrix}\right)$ and $B_1 = \left(\begin{matrix}
9 & 0 & 6 & 0\\
3 & 6 & 12 & 0
\end{matrix}\right)$.

Then, taking $\mu = 20$, we obtain $A_2 = \left(\begin{matrix}
0.25 & 0.2 & 0.5 & 0\\
0.35 & 0.1 & 0.4 & 0
\end{matrix}\right)$ and $B_2 = \left(\begin{matrix}
0.45 & 0 & 0.3 & 0\\
0.15 & 0.3 & 0.6 & 0
\end{matrix}\right)$.

It follows that $v^t = \frac{1}{n}(e^t - e^t B_2) = \frac{1}{2}((1 \ 1 \ 1 \ 1) - (0.6 \ 0.3 \ 0.9 \ 0)) = (0.2 \ 0.35 \ 0.05 \ 0.5)$.

Finally, $A' = \left(\begin{matrix}
0.45 & 0.55 & 0.55 & 0.5\\
0.55 & 0.45 & 0.45 & 0.5
\end{matrix}\right)$ and $B' = \left(\begin{matrix}
0.65 & 0.35 & 0.35 & 0.5\\
0.35 & 0.65 & 0.65 & 0.5
\end{matrix}\right)$.

Observe that $A', B' \in \Omega^{col}_{2, 4}$. Moreover, $A \preceq^s B$ if and only if $A' \preceq^s B'$.
\end{example}

A slightly different approach is the following.

\begin{reduction}\label{def:A':D}
Given $A, B \in M_{n, m}$ we construct column stochastic matrices $A', B' \in M_{n, m}$ in the following way:
\begin{enumerate}
    \item Choose $\lambda \geq 0$ large enough so that the matrices $A_1 = A + \lambda J$ and $B_1 + \lambda J$ are nonnegative and do not contain zero columns. By Corollary \ref{cor:+J}, $A \preceq^s B$ if and only if $A_1 \preceq^s B_1$.
    
    Note that $e^tA_1 = e^tA + n\lambda e^t$ and $e^tB_1 = e^tB + n\lambda e^t$.

    \item Let $D = diag(e^t B_1)$. That is, $D$ is the diagonal matrix, made up of the column sums of $B_1$. Note that by the definition of $B_1$, the diagonal entries of $D$ are strictly positive. The same is true for $D^{-1}$.
    
    Consider $A' = A_1D^{-1}$ and $B' = B_1D^{-1}$. By Corollary \ref{cor:AD<BD} we obtain that $A_1 \preceq^s B_1$ if and only if $A' \preceq^s B'$. Therefore $A \preceq^s B$ if and only if $A' \preceq^s B'$.

    Observe that $A'$ and $B'$ are nonnegative. In addition, $e^tB' = e^t$ and $$e^tA' = \left(\begin{array}{c|c|c}\frac{e^tA^{(1)} + n\lambda}{e^tB^{(1)} + n\lambda} \ & \ \ldots \ & \ \frac{e^tA^{(m)} + n\lambda}{e^tB^{(m)} + n\lambda}\end{array}\right).$$

    Therefore $B'$ is column stochastic. Moreover, $e^tA = e^tB$ if and only if $e^tA' = e^tB'$.
\end{enumerate}
\end{reduction}

Clearly, Corollary \ref{cor:strong} is also satisfied for Reduction \ref{def:A':D}. This reduction offers a more direct approach via the diagonal scaling than Reduction \ref{def:A':v}. On the other hand, Reduction \ref{def:A':v} is more flexible and provides a more convenient general expression for $e^tA'$.

\begin{example}
    Let us perform Reduction \ref{def:A':D} on $A = \left(\begin{matrix}
-1 & -2 & 4 & -6\\
1 & -4 & 2 & -6
\end{matrix}\right)$ and $B = \left(\begin{matrix}
3 & -6 & 0 & -6 \\
-3 & 0 & 6 & -6
\end{matrix}\right)$.

First, taking $\lambda = 7$ we obtain $A_1 = \left(\begin{matrix}
6 & 5 & 11 & 1\\
8 & 3 & 9 & 1
\end{matrix}\right)$ and $B_1 = \left(\begin{matrix}
10 & 1 & 7 & 1\\
4 & 7 & 13 & 1
\end{matrix}\right)$.

Then $D = \diag(14, 8, 20, 2)$ and we obtain $A' = \left(\begin{matrix}
\frac{6}{14} & \frac{5}{8} & \frac{11}{20} & \frac{1}{2}\\ \\ 
\frac{8}{14} & \frac{3}{8} & \frac{9}{20} & \frac{1}{2}
\end{matrix}\right)$ and $B' = \left(\begin{matrix}
\frac{10}{14} & \frac{1}{8} & \frac{7}{20} & \frac{1}{2}\\ \\ 
\frac{4}{14} & \frac{7}{8} & \frac{13}{20} & \frac{1}{2}
\end{matrix}\right)$.

Observe that $A', B' \in \Omega^{col}_{2, 4}$. Moreover, $A \preceq^s B$ if and only if $A' \preceq^s B'$.
\end{example}

\subsection{Weak majorization}

Note that in Lemma \ref{lem:ev^t} we only use the row stochasticity of $Q$. Therefore statements \ref{lem:lambda} --- \ref{cor:+J} remain true for weak majorization and we obtain

\begin{corollary}
    Let $A, B \in M_{n, m}$. Define a pair $A', B' \in M_{n, m}$ via Reduction \ref{def:A':v} or via Reduction \ref{def:A':D}. Then $A \preceq^w B$ if and only if $A' \preceq^w B'$. In addition, $B'$ is column stochastic. 
\end{corollary}

\begin{remark}
    Note that weak majorization does not imply that $e^t A = e^t B$. Therefore we cannot always assume that the conditions $A \preceq^w B$ and $B \in \Omega^{col}_{n, m}$ imply that $A \in \Omega^{col}_{n, m}$. For example, $\left(\begin{matrix}
1 & 1 \\
1 & 1
\end{matrix}\right) \preceq^w \left(\begin{matrix}
1 & 1 \\
0 & 0
\end{matrix}\right)$.

However, if in Reduction \ref{def:A':v}, Item \ref{item:A'}, we substitute $v = \frac{1}{n}(e^t - e^t B_2)$ with $w = \frac{1}{n}(e^t - e^t A_2)$, then $A'$ becomes column stochastic, while $B'$, in general, does not. Analogously, in Reduction \ref{def:A':D} we can substitute $D = \diag(e^t B_1)$ with $\diag (e^t A_1)$.
\end{remark}

\subsection{Directional majorization}

Statements \ref{lem:lambda} --- \ref{cor:+J} remain true for directional majorization, as we show below.

\begin{lemma}\label{lem:d:lambda}
    Let $A, B \in M_{n, m}$ and $\lambda \in \RR$, $\lambda \neq 0$. Then  $A \preceq^d B$ if and only if $\lambda A \preceq^d \lambda B$.
\end{lemma}

\begin{lemma}\label{lem:d:ev^t}
    Let $A, B \in M_{n, m}$ and $v \in \RR^m$. Then $A \preceq^d B$ if and only if $A + ev^t \preceq^d B + ev^t$.
\end{lemma}
\begin{proof}

    Let $A \preceq^d B$. Consider arbitrary $x \in \RR^m$. Then there exists $Q \in \Omega_n$ such that $Ax = QBx$. Observe that $QBx + Qe v^tx = Ax + e v^tx$. Therefore $(A + e v^t)x \preceq (B + e v^t)x$. As $x$ was arbitrary, we obtain $A + e v^t \preceq^d B + e v^t$.

    Now assume that $A + e v^t \preceq^d B + e v^t$. By the proven above, $A = (A + e v^t) - e v^t \preceq^d (B + e v^t) - e v^t = B$.
\end{proof}

\begin{corollary}\label{cor:d:+J}
    Let $A, B \in M_{n, m}$ and $\lambda \in \RR$. Then  $A \preceq^d B$ if and only if $A + \lambda J \preceq^d B + \lambda J$.
\end{corollary}

Therefore the same principles apply to directional majorization.

\begin{corollary}\label{cor:directional}
    Let $A, B \in M_{n, m}$. Define a pair $A', B' \in M_{n, m}$ via Reduction \ref{def:A':v} or via Reduction \ref{def:A':D}. Then $A \preceq^d B$ if and only if $A' \preceq^d B'$. In addition, $B'$ is column stochastic. 
    
    Since $e^t A = e^t B$ is a necessary condition for $A \preceq^d B$, we obtain that if the majorization holds, then $e^tA' = e^tB'$ and both $A', B'$ are column stochastic matrices.
\end{corollary}

\subsection{Vector majorization}\label{subsec:Reductions:vectors}

As a natural consequence of our matrix reduction techniques, vector majorization can be similarly reduced to the study of probability distributions.

More explicitly, consider $a, b \in \RR^n$. Assume that $e^t a = e^t b$, otherwise $a \not\preceq b$. Choose $\lambda \geq 0$ so that $a + \lambda e$ and $b + \lambda e$ are nonnegative and nonzero.

Consider $a' = \frac{1}{e^t a + n\lambda} (a + \lambda e)$ and $b' = \frac{1}{e^t a + n \lambda} (b + \lambda e)$. Then $a'$ and $b'$ are nonnegative vectors with $e^t a' = e^t b' = 1$.

Finally, similarly to Corollary \ref{cor:strong} we obtain that $a \preceq b$ if and only if $a' \preceq b'$.

\subsection{$(0, 1)$-matrices}

It is worth considering the case of $(0, 1)$-matrices separately. Majorization for $(0, 1)$-matrices can be easily reduced to column stochastic matrices with the help of the following map.

\begin{definition}\label{def:Theta}\rm
    Let us define an operator $\Theta$ on $M_{n, m}$ by $\Theta(X)^{(j)} =  \begin{cases}
        \frac{X^{(j)}}{e^t X^{(j)}}, \text{ if } X^{(j)} \neq 0,\\
        \frac{1}{n}e, \text{otherwise}
    \end{cases}$ for any $j \in \M$.
\end{definition}

\begin{remark}\label{rem:Theta:column}
    Let $A \in M_{n, m}(0, 1)$. Then every column of $\Theta(A)$ up to a permutation lies in the set $\{e_1, \frac{1}{2}e_1 + \frac{1}{2}e_2, \frac{1}{3}e_1 + \frac{1}{3}e_2 + \frac{1}{3}e_3, \ldots, \sum\limits_{q = 1}^{n} \frac{1}{n}e_q\}$.

    In particular, $\Theta(A) \in \Omega^{col}_{n, m}$.
\end{remark}

\begin{theorem}\label{thm:Theta}
    Let $A, B \in M_{n, m}(0, 1)$ with $e^tA = e^tB$. Then the following are equivalent:

    \begin{enumerate}
        \item $A \preceq^d B$;
        \item $A \preceq^s B$;
        \item $A = PB$ for some $P \in P(n)$;
        \item $\Theta(A) \preceq^d \Theta(B)$;
        \item $\Theta(A) \preceq^s \Theta(B)$;
    \end{enumerate}
\end{theorem}
\begin{proof}
    The equivalence of Items 1 --- 3 follows from the characterization of directional majorization for $(0, 1)$-matrices, see \cite[Theorem 3.5]{DGS2_01}

    Observe that as $e^t A = e^t B$, we conclude that $A^{(j)}=0$ if and only if $B^{(j)}=0$ for any $j \in \M$. Define $v \in \RR^m$ by $$\begin{cases}
        v_j = 1, \text{ if } A^{(j)} = B^{(j)} = 0;\\
        v_j = 0\text{, otherwise}.
    \end{cases}$$
    
    Define $A', B' \in M_{n, m}(0, 1)$ by $A' = A + ev^t$ and $B' = B + ev^t$.
     In other words, we substitute all zero columns of $A, B$ by the columns of all ones.

     Note that $\Theta(A) = \Theta(A')$ and $\Theta(B) = \Theta(B')$. Moreover, by Lemma \ref{lem:ev^t}, $A \preceq^s B$ if and only if $A' \preceq^s B'$. The same is true for directional majorization due to Lemma \ref{lem:d:ev^t}.

     Observe that $e^tA' = e^tB'$ and matrices $A', B'$ do not have zero columns. Let $D = \diag(e^tA') = \diag(e^tB') \in M_{m}$. Then the matrix $D$ is invertible. Moreover, $\Theta(A') = A'D^{-1}$ and $\Theta(B') = B'D^{-1}$. Then, by Corollary \ref{cor:AD<BD}, we obtain that $A' \preceq^d B'$ if and only if $\Theta(A) = \Theta(A') \preceq^d  \Theta(B') = \Theta(B)$, while $A' \preceq^s B'$ if and only if $\Theta(A) \preceq^s \Theta(B)$.
\end{proof}

\begin{corollary}
    Let $A, B \in M_{n, m}(0, 1)$ with $A \preceq^d B$. Then the following holds:

    \begin{enumerate}
        \item $A \preceq^s B$;
        \item $A = PB$ for some $P \in P(n)$;
        \item $\Theta(A) \preceq^d \Theta(B)$;
        \item $\Theta(A) \preceq^s \Theta(B)$;
    \end{enumerate}
\end{corollary}
\begin{proof}

    Let $A \preceq^d B$. Then $e^tA = e^tB$ and the result follows from Theorem \ref{thm:Theta}.
    
\end{proof}

The following example shows that $\Theta(A) \preceq^s \Theta(B)$ does not necessarily imply $A \preceq^s B$ for $A, B \in M_{n, m}(0, 1)$. This happens when the condition $e^tA = e^tB$ is not satisfied.
\begin{example}
    Let $A = \begin{pmatrix}
        1 \\ 1 \\ 1
    \end{pmatrix}, B = \begin{pmatrix} 1 \\ 1 \\ 0\end{pmatrix}$. Then $\Theta(A) = \begin{pmatrix}
        \frac{1}{3} \\ \frac{1}{3} \\ \frac{1}{3}
    \end{pmatrix}$ and $\Theta(B) = \begin{pmatrix}
        \frac{1}{2} \\ \frac{1}{2} \\ 0
    \end{pmatrix}$. Observe that $A \not\preceq^s B$, while $\Theta(A) = (\frac{1}{3}J) \Theta(B)$.
\end{example}

\begin{corollary}
    Let $\Phi$ be an operator on $M_{n, m}$. Assume that $\Phi$ preserves strong majorization for column stochastic matrices.

    Then $\Phi \circ \Theta$ preserves strong majorization for $(0, 1)$-matrices.
\end{corollary}
\begin{proof}
    Let $A, B \in M_{n, m}(0, 1)$ with $A \preceq^s B$. Then $\Theta(A) \preceq^s \Theta(B)$. Moreover, $\Theta(A), \Theta(B)\in \Omega^{col}_{n, m}$ by Remark \ref{rem:Theta:column}. It follows that $\Phi(\Theta(A))\preceq^s \Phi(\Theta(B))$.
\end{proof}

Linear operators preserving strong majorization for $(0, 1)$-matrices were characterized in \cite{GS3_preservers_01}. Note, however, that the operator $\Theta$ is not linear. Indeed, a map with column stochastic image can not be linear, since the zero matrix is not column stochastic.

Therefore, in general, we cannot characterize linear preservers of strong majorization for column stochastic matrices, via the result for $(0, 1)$-matrices. In the vector case, however, this reduction is possible, as we show in the next section.

\section{Linear operators preserving majorization for probability distributions}\label{sec:1}

In this section, we characterize linear operators preserving majorization for probability distributions. As we show in the lemma below, such operators necessarily preserve majorization for $(0, 1)$-vectors.

\begin{definition}\label{def:vectors} \rm
A linear operator $\phi$ on $\RR^n$ {\em preserves majorization for $(0, 1)$-vectors} if $a \preceq b$ implies $\phi(a) \preceq \phi(b)$ for any $a, b \in \{0, 1\}^n$.
\end{definition}

\begin{lemma}\label{lem:vector_case:01}
    Let $\phi$ be a linear operator on $\RR^n$. Assume that $\phi$ preserves majorization for probability distributions. Then $\phi$ preserves majorization for $(0, 1)$-vectors.
\end{lemma}
\begin{proof}
    Assume that $a, b \in \{0, 1\}^n$ and $a \preceq b$. Then, by Corollary \ref{cor:01:vector}, $e^ta = e^tb = k$ for some integer $0 \leq k \leq n$. If $k = 0$, then $a = b = 0$ and $\phi(a)=\phi(b) = 0$.

    If $k \neq 0$, then $a' = \frac{1}{k}a$ and $b'=\frac{1}{k}b$ are probability distributions. Therefore $\phi(a')\preceq \phi(b')$. But then $\phi(a)=k\phi(a')\preceq k \phi(b')=\phi(b)$.

    Therefore, $\phi$ preserves majorization for $(0, 1)$-vectors.
\end{proof}

The characterization of linear operators preserving majorization for $(0, 1)$-matrices was obtained in \cite{GS3_preservers_01}. We only provide here several relevant results that are concerned with the vector case.

\begin{theorem}\cite[Theorem 4.35]{GS3_preservers_01}\label{GS3:thm:vectors}
		Let $\phi$ be a linear operator on $\RR^n$, where $n \neq 3$. 
        
        Then $\phi$ preserves majorization for $(0, 1)$-vectors if and only if $\phi$ preserves vector majorization.
	\end{theorem}
	
		\begin{corollary}\cite[Corollary 4.36]{GS3_preservers_01}\label{GS3:cor:vectors}
			Let $\phi$ be a linear operator on $\RR^3$. Let $\alpha \in \RR \setminus \{0, 1\}$. 
            
            Then $\phi$ preserves vector majorization if and only if $\phi$ preserves majorization for $(0, 1)$-vectors and $\phi(\left(\begin{smallmatrix} \alpha \\ 1 \\ 0 \end{smallmatrix}\right)) \sim \phi(\left(\begin{smallmatrix} \alpha \\ 0 \\ 1 \end{smallmatrix}\right))$.
		\end{corollary}

\begin{theorem}\label{thm:vector}
    Let $\phi$ be a linear operator on $\RR^n$. 
    
    Then $\phi$ preserves majorization for probability distributions if and only if $\phi$ preserves vector majorization.
\end{theorem}
\begin{proof}
    If $\phi$ preserves vector majorization, then, in particular, it preserves majorization for probability distributions. Therefore we only need to prove the converse.

    Assume that $\phi$ preserves majorization for probability distributions. Then, by Lemma \ref{lem:vector_case:01}, $\phi$ preserves majorization for $(0, 1)$-vectors. If $n \neq 3$, then $\phi$ preserves vector majorization by Theorem \ref{GS3:thm:vectors}.

    It remains to consider the case $n = 3$. Let $a = \begin{pmatrix}
        2 \\ 1 \\ 0
    \end{pmatrix}, b = \begin{pmatrix}
        2 \\ 0 \\ 1
    \end{pmatrix}$. Then $\frac{1}{3}a\sim \frac{1}{3}b$ and these vectors are probability distributions. Therefore $\phi$ preserves the equivalence $\frac{1}{3}a\sim \frac{1}{3}b$. As a consequence, $\phi(a)\sim \phi(b)$. Then $\phi$ preserves vector majorization by Corollary \ref{GS3:cor:vectors}.
\end{proof}

\section{Linear operators preserving majorization for the zero-sum vectors}\label{sec:0}

In this section, we characterize linear preservers of majorization on $\0^n$, see Definition \ref{def:preserver:0}. Recall that $\0^n$ denotes the set of all real zero-sum vectors. This characterization turns out to be crucial for describing linear operators preserving majorization for column stochastic matrices.

\begin{definition}\label{def:preserver:0}\rm
    A linear operator $\phi$ on $\RR^n$ {\em preserves majorization on $\0^n$} if $a \preceq b$ implies $\phi(a)\preceq\phi(b)$ for any $a, b \in \0^n$.
\end{definition}

For the sake of brevity, we introduce the following terminology.
\begin{definition}\rm 
    A matrix $X\in M_{n, m}$ {\em satisfies condition $(0)$} if $Xa\sim XPa$ for any $a \in \0^m$ and any $P\in P(m)$.
\end{definition}

Consider the following necessary condition.

\begin{lemma}\label{lem:0:preserves->0}
    If a linear operator $\phi$ on $\RR^n$ preserves majorization on $\0^n$, then $[\phi]$ satisfies $(0)$.
\end{lemma}
\begin{proof}
    Let $a\in \0^n$ and $P\in P(n)$. Then $a\preceq Pa$ and $Pa\preceq a$. It follows that $[\phi]a=\phi(a)\sim \phi(Pa)=[\phi]Pa$.
\end{proof}

This condition is actually necessary and sufficient, as we will show in Lemma \ref{lem:0:prec-sim}.

In Statements \ref{lem:0:Xi-Xj_sim_Xk-Xl}---\ref{cor:0:Xi=Xj} we establish some column similarity properties of matrices satisfying $(0)$.

\begin{lemma}\label{lem:0:Xi-Xj_sim_Xk-Xl}
    Let $X \in M_n$ satisfy $(0)$. Then $X^{(i)} - X^{(j)} \sim X^{(k)} - X^{(l)}$ for any distinct $i, j \in \N$ and any distinct $k, l \in \N$.
\end{lemma}
\begin{proof}
    Observe that as $i\neq j$ and $k \neq l$ we obtain $e_i-e_j\sim e_k-e_l$. Moreover, as $e^t(e_i-e_j)=0$, we conclude that $X^{(i)} - X^{(j)} = X(e_i-e_j)\sim X(e_k-e_l) = X^{(k)} - X^{(l)}$.
\end{proof}

\begin{corollary}\label{cor:0:Xi-Xj_sim_Xj-Xi}
    Let $X \in M_n$ satisfy $(0)$. Then $X^{(i)} - X^{(j)} \sim X^{(j)} - X^{(i)}$ for any $i, j \in \N$.
\end{corollary}

\begin{corollary}\label{cor:0:e^tXi=e^tXj}
    Let $X \in M_n$ satisfy $(0)$. Then $e^t X^{(i)}=e^tX^{(j)}$ for any $i, j \in \N$.
\end{corollary}
\begin{proof}
    Let $i, j \in \N$. Then $X^{(i)} - X^{(j)} \sim X^{(j)} - X^{(i)}$ by Corollary \ref{cor:0:Xi-Xj_sim_Xj-Xi}. Therefore $e^t(X^{(i)} - X^{(j)})=e^t(X^{(j)} - X^{(i)})=-e^t(X^{(i)} - X^{(j)})$. Thus $e^t(X^{(i)} - X^{(j)})=0$ and $e^t X^{(i)}=e^tX^{(j)}$.
\end{proof}

\begin{corollary}\label{cor:0:Xi=Xj}
   Let $X \in M_n$ satisfy $(0)$. If $X^{(i)}=X^{(j)}$ for some distinct $i, j \in \N$, then $X = X^{(1)}e^t$.
\end{corollary}
\begin{proof}
    Consider arbitrary $k \in \N$, $k\neq i$. Then by Corollary \ref{cor:0:Xi-Xj_sim_Xj-Xi} we have $0=X^{(i)}-X^{(j)}\sim X^{(i)}-X^{(k)}$. Therefore $X^{(i)} = X^{(k)}$. It follows that $X = X^{(1)}e^t$, as $k$ was arbitrary.
\end{proof}

In Statements \ref{cor:0:e^tv=0=>e^tXv=0}---\ref{cor:0:alpha>0} we investigate the action of a matrix satisfying $(0)$ on $\0^n$.

\begin{corollary}\label{cor:0:e^tv=0=>e^tXv=0}
     Let $X \in M_n$ satisfy $(0)$. Then $Xv \in \0^n$ for any $v \in \0^n$.
\end{corollary}
\begin{proof}
    Observe that $e^tX=(e^tX^{(1)})e^t$ by Corollary \ref{cor:0:e^tXi=e^tXj}. Then $e^tXv=(e^tX^{(1)})e^tv=0$.
\end{proof}

The following lemma shows that we can always find a nonzero $v \in \0^n$ such that $Xv$ has at most two nonzero coordinates. Then the corollary above shows that these nonzero coordinates sum to zero.

\begin{lemma}\label{lem:0:alphae1-alphae2}
     Let $X \in M_n$ satisfy $(0)$. Then for any distinct $i, j \in \N$ there exists a nonzero $v \in \0^n$ such that $Xv=\alpha(e_i - e_j)$ for some $\alpha \geq 0$.
\end{lemma}
\begin{proof}
    If $n=2$, then let $v = \begin{pmatrix}
        1\\-1
    \end{pmatrix}$. It follows from Corollary \ref{cor:0:e^tv=0=>e^tXv=0} that $Xv=\begin{pmatrix}
        \beta \\ -\beta
    \end{pmatrix}$ for some $\beta \in \RR$. If $(Xv)_i \geq 0$, then we have found the desired $v$. Otherwise $-(Xv)_i \geq 0$ and $-v$ is the required vector.
    
    For $n > 2$ consider the $(n-2)\times n$ submatrix $X'$ obtained by deleting the rows $X_{(i)}$ and $X_{(j)}$ from $X$. Let $X''=\begin{pmatrix}
        e^t\\X'
    \end{pmatrix}\in M_{n-1,n}$.

    The columns of $X''$ are linearly dependent. Therefore there exists $v\in \RR^n$ with $X''v=0$. Then $e^tv=0$ and $X'v=0$. It follows that $Xv = \alpha e_i + \beta e_j$ for some $\alpha, \beta \in \RR$. In addition, $\beta = -\alpha$ by Corollary \ref{cor:0:e^tv=0=>e^tXv=0}.

    If $\alpha \geq 0$, then we have found the desired $v$. Otherwise $-\alpha\geq 0$ and $-v$ is the required vector.
\end{proof}

The lemma above shows that there exists such $v\in \0^n$ that $Xv$ has at most two nonzero coordinates. The case $Xv=0$ is investigated in the following corollary.

\begin{corollary}\label{cor:0:Xv=0}
    Let $X \in M_n$ satisfy $(0)$. Assume that there exists a nonzero $v\in \0^n$ such that $Xv=0$. Then $X = X^{(1)}e^t$.
\end{corollary}
\begin{proof}
    Since $v \in \0^n$ and it is nonzero, there exist $i, j\in \N$ such that $v_i \neq v_j$.

    Then $P_{(ij)}v=v-(v_i-v_j)e_i+(v_i-v_j)e_j$ and $0=Xv\sim XP_{(ij)}v=Xv-(v_i-v_j)X^{(i)}+(v_i-v_j)X^{(j)}=(v_i-v_j)(X^{(j)}-X^{(i)})$.

    However, $v_i-v_j\neq 0$. It follows that $X^{(i)}=X^{(j)}$. Then $X=X^{(1)}e^t$ by Corollary \ref{cor:0:Xi=Xj}.
\end{proof}

As a simple consequence, we obtain

\begin{corollary}\label{cor:0:alpha>0}
    Let $X \in M_n$ satisfy $(0)$. Assume that $X^{(i)}\neq X^{(j)}$ for some $i, j \in \N$. Then there is no nonzero $v\in \0^n$ such that $Xv=0$. In particular, $\alpha > 0$ in Lemma \ref{lem:0:alphae1-alphae2}.
\end{corollary}
\begin{proof}
    As $X^{(i)}\neq X^{(j)}$, we obtain that $X \neq X^{(1)}e^t$. Then the rest follows by Corollary \ref{cor:0:Xv=0}.
\end{proof}

We now prove the following technical lemma.

\begin{lemma}\label{lem:0:alpha+-lambda_y}
    Let $y\in \RR^n$, $\alpha > 0$. If $\alpha(e_1-e_2)-\lambda y\sim \alpha(e_1-e_2) + \lambda y$ for any $\lambda > 0$, then $y_1=y_2=0$.
\end{lemma}
\begin{proof}

    Denote $a = \alpha(e_1-e_2)-\lambda y, b = \alpha(e_1-e_2)+\lambda y$.

    Let us choose a sufficiently small $\lambda > 0$ so that $\alpha > 2\lambda\max\limits_{i}|y_i|$.

    Then $\alpha - \lambda y_1 > 2\lambda\max\limits_{i}|y_i| - \lambda\max\limits_{i}|y_i|=\lambda\max\limits_{i}|y_i| > -\lambda y_q$ for any $q \in \N$.

    Also $\alpha - \lambda y_1 > -\alpha -\lambda y_2$.

    It follows that $\alpha - \lambda y_1 = \max(a)$.

    Similarly, $\alpha + \lambda y_1 > 2\lambda\max\limits_{i}|y_i| - \lambda\max\limits_{i}|y_i|=\lambda\max\limits_{i}|y_i| > \lambda y_q$ for any $q \in \N$ and $\alpha + \lambda y_1>-\alpha +\lambda y_2$. Therefore $\alpha + \lambda y_1 = \max(b)$.

    On the other hand, $\max(a)=\max(b)$ since $a\sim b$. Then $\alpha - \lambda y_1 = \alpha + \lambda y_1$, while $\lambda > 0$. It follows that $y_1 = 0$.

    Considering vectors $-a$ and $-b$ we similarly obtain that $y_2=0$.
\end{proof}

For $n \geq 5$ we can establish an upper bound on the number of possible distinct values in a matrix satisfying $(0)$.

\begin{lemma}\label{lem:0:>=5}
    Let $X \in M_n$, $n \geq 5$, satisfy $(0)$. Then for any $i, j\in \N$ the columns $X^{(i)}$ and $X^{(j)}$ can have distinct values only in at most three rows. That is, there are at most three nonzero coordinates in $X^{(i)}-X^{(j)}$.
\end{lemma}
\begin{proof}
    Assume that $X^{(i)}$ and $X^{(j)}$ have distinct values in at least four rows. Without loss of generality, $x_{1i}\neq x_{1j}, x_{2i}\neq x_{2j}, x_{3i}\neq x_{3j}, x_{4i}\neq x_{4j}$.

    By Lemma \ref{lem:0:alphae1-alphae2}, there exists a nonzero $v \in \0^n$ such that $Xv=\alpha(e_4-e_5)$ for some $\alpha \geq 0$. Moreover, $\alpha > 0$ by Corollary \ref{cor:0:alpha>0}.
    
    There are two possibilities:

    \textbf{Case 1.} $v_i\neq v_j$. Observe that $P_{(ij)}v = v -(v_i-v_j)(e_i - e_j)$. Since $e^tv=e^tP_{(ij)}v=0$ and $v\sim P_{(ij)}v$, we obtain that $$\alpha(e_4-e_5)=Xv\sim XP_{(ij)}v=Xv-(v_i-v_j)(X^{(i)}-X^{(j)})=\alpha(e_4-e_5)-(v_i-v_j)(X^{(i)}-X^{(j)}).$$

    On the other hand, for $q\in \{1, 2, 3\}$ we observe $(XP_{(ij)}v)_q=(v_j-v_i)(x_{qi}-x_{qj})\neq 0$. This contradicts $XP_{(ij)}v\sim Xv=\alpha(e_4-e_5)$ because the latter vector has only two nonzero coordinates. Therefore, this case is impossible.

    \textbf{Case 2.} $v_i = v_j$. Consider arbitrary $\lambda > 0$ and $u = v - \lambda (e_i - e_j), w = v + \lambda (e_i - e_j)$. Observe that $e^t u = e^t w = e^t v = 0$. In addition, $\begin{cases}
        u_i = v_i - \lambda = v_j - \lambda = w_j;\\
        u_j = v_j + \lambda = v_i + \lambda = w_i;\\
        u_q = w_q = v_q \text{ if } q\neq i, j.
    \end{cases}$ That is, $u=P_{(ij)}w$.
    
    Therefore, $Xu\sim Xw$. Moreover, $\begin{cases}
        Xu = Xv - \lambda(X^{(i)} - X^{(j)});\\
        Xw = Xv + \lambda(X^{(i)} - X^{(j)}).
    \end{cases}$

    Finally, for any $\lambda > 0$ we have $\alpha(e_4-e_5) - \lambda(X^{(i)} - X^{(j)}) \sim \alpha(e_4-e_5) + \lambda(X^{(i)} - X^{(j)})$. Then $(X^{(i)} - X^{(j)})_4 = 0$ by Lemma \ref{lem:0:alpha+-lambda_y}, a contradiction.
\end{proof}

The following simple observation allows us to refine the lemma above in Corollary \ref{cor:0:Xi-Xj_sim_alpha(e_1-e_2)}.

\begin{lemma}\label{lem:0:even_number}
    Let $v \in \RR^n$ and $v \sim -v$. Then the number of the nonzero coordinates of $v$ is even.
\end{lemma}
\begin{proof}
    As $v \sim -v$, the number of the positive coordinates of $v$ and $-v$ coincides. That is, the number of the positive coordinates of $v$ coincides with the number of the negative coordinates of $v$. It follows that the number of the nonzero coordinates of $v$ is even.
\end{proof}

\begin{corollary}\label{cor:0:Xi-Xj_sim_alpha(e_1-e_2)}
    Let $X\in M_n$, $n \geq 5$, satisfy $(0)$. 
    
    Then either $X=X^{(1)}e^t$, or there exists $\alpha > 0$ such that $X^{(i)} - X^{(j)} \sim \alpha(e_1 - e_2)$ for any $i, j \in \N$.
\end{corollary}
\begin{proof}
    Consider the columns $X^{(1)}$ and $X^{(2)}$. According to Lemma \ref{lem:0:>=5}, the vector $X^{(1)}-X^{(2)}$ can have at most three nonzero coordinates. On the other hand, $X^{(1)}-X^{(2)}\sim X^{(2)}-X^{(1)}$ by Corollary \ref{cor:0:Xi-Xj_sim_Xj-Xi}. Then it follows from Lemma \ref{lem:0:even_number} that either $X^{(1)}=X^{(2)}$, or  the vector $X^{(1)}-X^{(2)}$ has exactly two nonzero coordinates.

    If $X^{(1)}=X^{(2)}$, then $X = X^{(1)}e^t$ by Corollary \ref{cor:0:Xi=Xj}. 
    
    Assume that $X^{(1)}-X^{(2)}$ has exactly two nonzero coordinates. By Corollary \ref{cor:0:e^tXi=e^tXj} $e^t(X^{(1)}-X^{(2)}) = 0$ and thus $X^{(1)}-X^{(2)} \sim \alpha(e_1-e_2)$ for some $\alpha > 0$.

    Finally, by Lemma \ref{lem:0:Xi-Xj_sim_Xk-Xl}, $\alpha(e_1-e_2)\sim X^{(1)}-X^{(2)} \sim X^{(i)}-X^{(j)}$.
\end{proof}

We can generalize the result above to smaller matrices.

\begin{lemma}\label{lem:0:>=3}
Let $X\in M_n$, $n \geq 3$, satisfy $(0)$. 

Then either $X = X^{(1)}e^t$ or there exists $\alpha > 0$ such that $X^{(i)} - X^{(j)} \sim \alpha(e_1 - e_2)$ for any $i, j \in \N$. In particular, in the latter case any two different columns $X^{(i)}, X^{(j)}$ have distinct values in exactly two rows.
\end{lemma}
\begin{proof}
    Due to Corollary \ref{cor:0:Xi-Xj_sim_alpha(e_1-e_2)} we only need to consider $n = 3$ and $n = 4$.

    Observe that $X^{(1)} -X^{(2)}\sim X^{(2)}-X^{(1)}$ by Corollary \ref{cor:0:Xi-Xj_sim_Xj-Xi}. Therefore Lemma \ref{lem:0:even_number} states that the number of nonzero coordinates in $X^{(1)}-X^{(2)}$ is even. Note also that $e^tX^{(1)}=e^tX^{(2)}$ by Corollary \ref{cor:0:e^tXi=e^tXj}.
    
    If $X^{(1)}=X^{(2)}$, then $X = X^{(1)}e^t$ by Corollary \ref{cor:0:Xi=Xj}. If $X^{(1)}-X^{(2)}$ has exactly two nonzero coordinates, then $X^{(1)}-X^{(2)}\sim \alpha(e_1 - e_2)$ for some $\alpha > 0$. Moreover, by Lemma \ref{lem:0:Xi-Xj_sim_Xk-Xl}, we obtain that $X^{(i)} - X^{(j)}\sim X^{(1)}-X^{(2)}\sim \alpha(e_1-e_2)$.

    Assume that $X^{(1)}-X^{(2)}$ has four nonzero coordinates. Naturally, this can only happen if $n = 4$. Therefore $X^{(1)}-X^{(2)}$ has no zero coordinates. Then by Lemma \ref{lem:0:Xi-Xj_sim_Xk-Xl} there are no zero coordinates in $X^{(i)}-X^{(j)}$ for any distinct $i, j \in \N$. Therefore the entries $x_{11}, x_{12}, x_{13}$ and $x_{14}$ are distinct. Without loss of generality, assume that $x_{11}<x_{12}<x_{13}<x_{14}$.

    As $X^{(i)}-X^{(j)}\sim X^{(k)}-X^{(l)}$ for any distinct $i, j \in \N$ and any distinct $k, l \in \N$, the distinct values $x_{11} - x_{14} < x_{11}-x_{13} < x_{11}-x_{12} < x_{12} - x_{11} < x_{13} - x_{11} < x_{14}-x_{11}$ must be the entries of $X^{(1)}-X^{(2)}$. But this contradicts $n = 4$.
\end{proof}

\subsection{Matrices satisfying $(\alpha)$}

Consider the following property of a matrix $X \in M_n$.
\begin{definition}\label{def:alpha}\rm 
    A matrix $X \in M_n$ {\em satisfies condition $(\alpha)$} if there exists $\alpha > 0$ such that $X^{(i)}-X^{(j)}\sim \alpha(e_1 - e_2)$ for any distinct $i, j \in \N$.
\end{definition}

Lemma \ref{lem:0:>=3} states that if $X \in M_n$, $n \geq 3$, satisfies $(0)$ and $X \neq X^{(1)}e^t$, then $X$ satisfies $(\alpha)$. In this section, we are going to characterize all square matrices satisfying $(\alpha)$. As we shall see in Theorems \ref{thm:alpha} and \ref{thm:0}, this property is very similar to preserving majorization for the zero-sum vectors.

\begin{lemma}\label{lem:alpha:e^tXi=e^tXj}
    Let $X \in M_n$ satisfy $(\alpha)$. Then $e^t X^{(i)}=e^tX^{(j)}$ for any $i, j \in \N$.
\end{lemma}
\begin{proof}
    Observe that $e^t(X^{(i)}- X^{(j)})=e^t(\alpha(e_1 - e_2)) = 0$.
\end{proof}

Next we show that in every row of $X$ there are at most two distinct values.

\begin{lemma}\label{lem:alpha:2distinct_in_a_row}
    Let $X \in M_n$ satisfy $(\alpha)$. Then for any $i \in \N$ there are at most two distinct entries in $X_{(i)}$.
\end{lemma}
\begin{proof}
    Assume that for some $i, j, k, l \in \N$ we have $x_{ij} < x_{ik} < x_{il}$. In this case, $x_{il} - x_{ij}$ and $x_{ik} - x_{ij}$ are distinct positive numbers. But this contradicts $X^{(l)}-X^{(j)}\sim X^{(k)}-X^{(j)}\sim \alpha(e_1 - e_2)$.
\end{proof}

The following technical lemma shows that for some fixed columns $k, l, m$ the situation $x_{ik} \neq x_{il}=x_{im}$ can only happen in one row of $X$.

\begin{lemma}\label{lem:alpha:technical}
    Let $X \in M_n$ satisfy $(\alpha)$. Assume that there are $i \in \N$ and distinct $k, l, m \in \N$ such that $x_{ik} \neq x_{il}=x_{im}$. Then $x_{jk} \in \{x_{jl}, x_{jm}\}$ for any $j \in \N$, $j \neq i$.
\end{lemma}
\begin{proof}
    Assume that $x_{jk}\neq x_{jl}, x_{jm}$. Then, by Lemma \ref{lem:alpha:2distinct_in_a_row}, $x_{jl}=x_{jm}$. As $x_{ik}\neq x_{il}$, $x_{jk}\neq x_{jl}$ and $X$ satisfies $(\alpha)$, we obtain that $x_{qk}=x_{ql}$ for any $q \in \N$, $q \neq i, j$. Similarly, $x_{qk}=x_{qm}$ for any $q \in \N$, $q \neq i, j$. It follows that $X^{(l)} = X^{(m)}$ and $X^{(l)}-X^{(m)}=0$, a contradiction.
\end{proof}

The lemma below provides the key feature of matrices satisfying $(\alpha)$.

 \begin{lemma}\label{lem:alpha:n-1/n}
     Let $X \in M_n$ satisfy $(\alpha)$. Then in every row $X_{(i)}$ of $X$ at least $n-1$ of the $n$ entries are equal.
 \end{lemma}
 \begin{proof}
     By Lemma \ref{lem:alpha:2distinct_in_a_row} in every row of $X$ there are at most two distinct entries. For $n \leq 3$ this immediately proves the required result.

    Let $n \geq 4$. Assume the contrary, that, without loss of generality, $x_{11}=x_{12}\neq x_{13}=x_{14}$. The vector $X^{(1)}-X^{(2)}$ has exactly two nonzero coordinates. Without loss of generality, we may assume that $x_{21}\neq x_{22}$ and $x_{31}\neq x_{32}$. Then $X = \begin{pmatrix}
        x_{11} & x_{11} & x_{13} & x_{13} & \cdots\\
        x_{21} & x_{22} & \cdots\\
        x_{31} & x_{32} & \cdots\\
        \vdots
    \end{pmatrix}$.

    Applying Lemma \ref{lem:alpha:technical} to rows $1, 2$ and columns $1, 3, 4$ we obtain that either $x_{21} = x_{23}$ or $x_{21} = x_{24}$. Similarly, applying that lemma to the same rows and to columns $2, 3, 4$ we obtain that either $x_{22} = x_{23}$ or $x_{22} = x_{24}$.

    Permuting, if necessary, columns $X^{(3)}$ and $X^{(4)}$ we obtain $X = \begin{pmatrix}
        x_{11} & x_{11} & x_{13} & x_{13} & \cdots\\
        x_{21} & x_{22} & x_{21} & x_{22} & \cdots\\
        x_{31} & x_{32} & \cdots\\
        \vdots
    \end{pmatrix}$. Observe that $X^{(1)}$ and $X^{(4)}$ have distinct values in rows $1, 2$. Therefore $x_{31}=x_{34}$. Similarly, $X^{(2)}$ and $X^{(3)}$ have distinct values in rows $1, 2$. Therefore $x_{32}=x_{33}$.

    It follows that $X = \begin{pmatrix}
        x_{11} & x_{11} & x_{13} & x_{13} & \cdots\\
        x_{21} & x_{22} & x_{21} & x_{22} & \cdots\\
        x_{31} & x_{32} & x_{32} & x_{31} & \cdots\\
        \vdots
    \end{pmatrix}$ and $$X^{(1)}-X^{(2)}=(x_{21}-x_{22})e_2 + (x_{31}-x_{32})e_3\sim X^{(3)}-X^{(4)}=(x_{21}-x_{22})e_2 + (x_{32}-x_{31})e_3.$$  But this is impossible since $x_{32}-x_{31}\neq 0$.
 \end{proof}
 
The following theorem characterizes square matrices satisfying $(\alpha)$. In the next section we will show that this essentially characterizes linear operators preserving majorization for the zero-sum vectors.
 
 \begin{theorem}\label{thm:alpha}
     Let $X \in M_n$. Then $X$ satisfies $(\alpha)$ if and only if $X = ve^t + \lambda P$ for some $\lambda \neq 0$, $v \in \RR^n$ and $P\in P(n)$.
 \end{theorem}
 \begin{proof}

    Assume that $X$ satisfies $(\alpha)$.
    
     First of all, note that due to Lemma \ref{lem:alpha:n-1/n} in every row of $X$ at least $n-1$ of the $n$ entries coincide. Let us call the other entries {\em `unique'}. In other words, an entry $x_{ij}$ of $X_{(i)}$ is unique if $x_{ij}\neq x_{i1} = \ldots = x_{i,j-1}=x_{i,j+1} = \ldots = x_{in}$. There are at most $n$ unique entries in all of $X$. Note that, in general, the unique entries of different rows do not have to be distinct.

     By Lemma \ref{lem:alpha:technical}, no column of $X$ can have more than one unique entry. On the other hand, if there are two columns $X^{(j_1)}$, $X^{(j_2)}$ without unique entries, then $X^{(j_1)}-X^{(j_2)}=0$, which contradicts Condition $(\alpha)$.

     If there is only one column $X^{(j_1)}$ without unique entries, then any other column $X^{(j_2)}$ must have exactly one unique entry. But in this case the vector $X^{(j_1)}-X^{(j_2)}$ has exactly one nonzero entry, which contradicts Condition $(\alpha)$.

     Finally, we conclude that there is exactly one unique entry in every column of $X$. Combining this with the fact that in every row there is at most one unique entry, we obtain that there exists $P \in P(n)$ such that $x_{ij}$ is unique if and only if $p_{ij}=1$.

     Let $v \in \RR^n$ be such that $v_i$ is the non-unique entry value of $X_{(i)}$, $i \in \N$. That is, $x_{ij} = v_i$ for any $i, j \in \N$ with $p_{ij}=0$.

     Consider $D = X - v e^t$. Recall that by Lemma \ref{lem:alpha:e^tXi=e^tXj} $e^tX = (e^t X^{(1)})e^t$. Then $e^t D = e^t X - e^t v e^t = e^t (X^{(1)} - v)e^t = \lambda e^t$, where $\lambda = e^t(X^{(1)} - v)$. 
     
     By the definitions of $v$ and $P$ we obtain that $d_{ij} \neq 0$ if and only if $p_{ij}=1$. Then for any $i, j \in \N$ with $p_{ij}=1$ we obtain $d_{ij} = (e^t D)_j = \lambda$. Therefore $D = \lambda P$.

     Finally, $X = ve^t + \lambda P$. In addition, $\lambda \neq 0$, because otherwise $X^{(1)}=X^{(2)}$, which contradicts $(\alpha)$.

    Now to prove the converse, assume that $X = ve^t + \lambda P$ for some $\lambda \neq 0$, $v \in \RR^n$ and $P\in P(n)$. Then $X^{(i)} - X^{(j)} = \lambda (P^{(i)} - P^{(j)}) \sim \lambda(e_1 - e_2)$.
 \end{proof}

\subsection{Characterization of linear operators preserving majorization for the zero-sum vectors}\label{subsec:0^n:preservers}

In this section, we characterize linear operators preserving majorization on $\0^n$. We start by proving the following sufficient condition.

\begin{lemma}\label{lem:0:sufficient}
    Let $\phi$ be a linear operator on $\RR^n$ given by $[\phi] = ve^t + \lambda P$ for some $v \in \RR^n$, $\lambda \in \RR$ and $P \in P(n)$.

    Then $\phi$ preserves majorization on $\0^n$.
\end{lemma}
\begin{proof}
    Let $a, b \in \0^n$ with $a \preceq b$.

    Then $\phi(a) = ve^ta + \lambda P a = \lambda P a$. Dually, $\phi(b) = \lambda P b$. It follows that $\phi(a) \preceq \phi(b)$.
\end{proof}

Before providing the general characterization, we treat the case $n = 2$ separately.

\begin{lemma}\label{lem:0:n=2}
    Let $\phi$ be a linear operator on $\RR^2$. Then $\phi$ preserves majorization on $\0^2$ if and only if $[\phi] = ve^t + \lambda P$ for some $v \in \RR^2$, $\lambda \in \RR$ and $P \in P(2)$.
\end{lemma}
\begin{proof}
    Assume that $\phi$ preserves majorization on $\0^n$. Then $X = [\phi]$ satisfies $(0)$ by Lemma \ref{lem:0:preserves->0}. It follows that $e^t X^{(1)}=e^tX^{(2)}$ by Corollary \ref{cor:0:e^tXi=e^tXj}. It means that $x_{22} = x_{11} + x_{21} - x_{12}$ and $$X = \begin{pmatrix}
        x_{11} & x_{12} \\ x_{21} & x_{11} + x_{21} - x_{12}
    \end{pmatrix} = \begin{pmatrix}
        x_{12} \\ x_{21}
    \end{pmatrix} (1 \ 1) + (x_{11} - x_{12})I.$$

    Lemma \ref{lem:0:sufficient} concludes the proof.
\end{proof}

\begin{theorem}\label{thm:0}
    Let $\phi$ be a linear operator on $\RR^n$. Then $\phi$ preserves majorization on $\0^n$ if and only if $$[\phi] = ve^t + \lambda P \text{ for some } v \in \RR^n, \lambda \in \RR \text{ and } P \in P(n).$$
\end{theorem}
\begin{proof}
    The case $n = 2$ was settled in Lemma \ref{lem:0:n=2}. Thus we only consider $n \geq 3$.

    Assume that $\phi$ preserves majorization on $\0^n$. Then $X = [\phi]$ satisfies $(0)$ by Lemma \ref{lem:0:preserves->0}. By Lemma \ref{lem:0:>=3}, one of the following holds:
    \begin{enumerate}
        \item $X = X^{(1)}e^t = X^{(1)}e^t + 0 I$.

        \item $X$ satisfies $(\alpha)$. In this case, $[\phi] = X = ve^t + \lambda P$ for some $v \in \RR^n$, $\lambda \neq 0$ and $P \in P(n)$ by Theorem \ref{thm:alpha}.
    \end{enumerate}

    Lemma \ref{lem:0:sufficient} concludes the proof.
\end{proof}

    Operators in Theorem \ref{thm:0} provide a better understanding of the nature of Ando's operators, see Theorem \ref{Ando}. The two types of Ando's operators are just particular cases of the natural general structure $ve^t + \lambda P$ that describes linear preservers of majorization on $\0^n$. Namely, Ando's operators of the first type are obtained by letting $\lambda = 0$, while Ando's operators of the second type are obtained by letting $v = \gamma e$ for some $\gamma \in \RR$.

\section{Linear operators preserving strong majorization for column stochastic matrices}\label{sec:Result}

In this section, we give a complete characterization of linear operators preserving strong majorization for column stochastic matrices. We consider a standard decomposition of a linear operator on $M_{n, m}$ into $m^2$ linear operators on $\RR^n$.

\begin{definition}\label{def:decomposition}\rm
    For a linear operator $\Phi$ on $M_{n, m}$ we consider its decomposition into $m^2$ linear operators $\Phi^j_i$, $i, j\in \M$ on $\RR^n$:

$$\Phi(X) = \left(\begin{array}{c|c|c}\sum\limits_{j=1}^{m}\Phi^j_1(X^{(j)}) \ & \ \ldots \ & \ \sum\limits^{m}_{j=1}\Phi^j_m(X^{(j)})\end{array}\right).$$
    
\end{definition}

This is a standard decomposition, but for the sake of completeness, we provide the following

\begin{lemma}\label{lem:Phi:decomposition}
    The decomposition of $\Phi$ into operators $\Phi^j_i$ always exists and it is unique. Namely, $\Phi^j_i(v)=\Phi(ve_j^t)e_i$ for any $v \in \RR^n$.
\end{lemma}
\begin{proof}
    For any $i, j \in \M$ consider $\Phi^j_i$, defined by $\Phi^j_i(v)=\Phi(ve_j^t)e_i$. Then, due to the linearity of $\Phi$, we obtain 
    
    $$\Phi(X) = \sum\limits_{j = 1}^{m}\Phi(X^{(j)}e^t_j) =  \left(\begin{array}{c|c|c}\sum\limits_{j = 1}^{m}\Phi(X^{(j)}e^t_j)^{(1)} \ & \ \ldots \ & \
    \sum\limits_{j = 1}^{m}\Phi(X^{(j)}e^t_j)^{(m)}\end{array}\right) = \left(\begin{array}{c|c|c}\sum\limits_{j=1}^{m}\limits\Phi^j_1(X^{(j)}) \ & \ \ldots \ & \ \sum^{m}_{j=1}\limits\Phi^j_m(X^{(j)})\end{array}\right),$$ which gives the desired decomposition.

    On the other hand, consider an arbitrary decomposition $$\Phi(X) = \left(\begin{array}{c|c|c}\sum\limits_{j=1}^{m}\limits\Phi^j_1(X^{(j)}) \ & \ \ldots \ & \ \sum^{m}_{j=1}\limits\Phi^j_m(X^{(j)})\end{array}\right).$$

    Then for any $v\in \RR^n$ and any $i, j \in \M$ we have $$\Phi(ve^t_j)e_i = \left(\begin{array}{c|c|c}\Phi^j_1(v)  \ & \  \ldots  \ & \  \Phi^j_m(v)\end{array}\right)e_i = \Phi^j_i(v).$$

    Therefore, the decomposition is unique.
\end{proof}

Assume that $\Phi$ preserves strong majorization for column stochastic matrices. As it turns out, we cannot say that each $\Phi^j_i$ preserves majorization for probability distributions. An analogous reduction was possible from preservers of strong majorization to preservers of vector majorization, see \cite[Theorem 2]{LiPoon2001}. Also, from preservers of strong majorization for $(0, 1)$-matrices to majorization for $(0, 1)$-vectors, see \cite[Lemma 5.5]{GS3_preservers_01}. The same approach fails here, because it relies on matrices with exactly one nonzero column, see \cite{LiPoon2001}. For $m > 1$ these matrices are not column stochastic and the operator $\Phi$ does not have to preserve majorization for such matrices. We shall indeed see in Example \ref{ex:last} that $\Phi^j_i$ does not necessarily preserve majorization for probability distributions. 

However, we will discover that every $\Phi^j_i$ preserves majorization for the zero-sum vectors. This will be shown in Section \ref{subsec:Phi^j_k:preserves:0}. Moreover, in the following lemma we show that the operators $\sum\limits_{j = 1}^{m} \Phi^j_k$ do preserve majorization on $\1^n$. Recall that due to Theorem \ref{thm:vector} this is equivalent to preserving vector majorization in general.

\begin{lemma}\label{lem:1:sum_preserves}
    Let $\Phi$ be a linear operator on $M_{n, m}$ preserving strong majorization on $\Omega^{col}_{n, m}$. Then $\sum\limits_{j = 1}^{m} \Phi^j_k$ preserves vector majorization for any $k \in \M$.
\end{lemma}
\begin{proof}
    Consider arbitrary $a, b \in \1^n$ with $a \preceq b$. Let $A = ae^t, B = be^t$. Then $A \preceq^s B$. Therefore $\Phi(A) \preceq^s \Phi(B)$, as $A, B \in \Omega^{col}_{n, m}$. In particular, $\Phi(A)^{(k)}\preceq \Phi(B)^{(k)}$.

    Further, $\Phi(A)^{(k)} = \sum\limits_{j = 1}^m\Phi^j_k(A^{(j)}) = (\sum\limits_{j = 1}^m\Phi^j_k)(a)$. Dually, $\Phi(B)^{(k)} = (\sum\limits_{j = 1}^m\Phi^j_k)(b)$. Then $(\sum\limits_{j = 1}^m\Phi^j_k)(a) \preceq (\sum\limits_{j = 1}^m\Phi^j_k)(b)$. It follows that $\sum\limits_{j = 1}^{m} \Phi^j_k$ preserves majorization on $\1^n$. Finally, $\sum\limits_{j = 1}^{m} \Phi^j_k$ preserves vector majorization by Theorem \ref{thm:vector}.
\end{proof}

\subsection{Every $\Phi^j_k$ preserves majorization on $\0^n$}\label{subsec:Phi^j_k:preserves:0}

We start by observing that the necessary condition in Lemma \ref{lem:0:preserves->0} is actually sufficient.

\begin{lemma}\label{lem:0:prec-sim}
    Let $\phi$ be a linear operator on $\RR^n$. Then $\phi$ preserves majorization on $\0^n$ if and only if $\phi(a)\sim \phi(Pa)$ for every $a \in \0^n$ and every $P \in P(n)$ .
\end{lemma}
\begin{proof}

    The necessity follows from Lemma \ref{lem:0:preserves->0}. We only need to prove the sufficiency.

    Assume that $\phi(x)\sim \phi(Px)$ for any $x \in \0^n$ and $P \in P(n)$. Consider arbitrary $a, b \in \0^n$ with $a \preceq b$. Our goal is to prove that $\phi(a)\preceq \phi(b)$. Observe that $a = Qb$ for some $Q \in \Omega_n$ by Theorem \ref{HLP}.

    On the other hand, by Theorem \ref{Birkhoff}, $$Q = \lambda_1 P_1 + \ldots \lambda_k P_k \text{ for some } (\lambda_1 \ \ldots \ \lambda_k)^t \in \1^k \text{ and } P_1, \ldots, P_k\in P(n).$$ It follows that $a = \sum\limits_{i=1}^{k}\lambda_i P_i b$.

    Due to the conditions on $\phi$ we obtain that for any $i \in \{1, \ldots, k\}$ there exists $P'_i \in P(n)$ such that $\phi(P_ib) = P'_i\phi(b)$. Therefore, $\phi(a) = (\sum\limits_{i=1}^{k} \lambda_i P'_i) \phi(b)$.

    As $(\lambda_1, \ \ldots, \lambda_k)^t \in \1^k$, we conclude that $\sum\limits_{i=1}^{k} \lambda_i P'_i \in \Omega_n$ and thus $\phi(a) \preceq \phi(b)$. Finally, $\phi$ preserves vector majorization on $\0^n$.
\end{proof}

In Lemma \ref{lem:s+phi->preserves_0<=1} and Corollary \ref{cor:s+phi->preserves_0<=1/n} we provide new sufficient conditions for preserving majorization on $\0^n$. The following technical lemma allows us to do so.

Recall that $a^+$ denotes the sum of the positive entries of $a \in \RR^n$. That is, $a^+ = \sum\limits_{i=1}^n \max(a_i, 0)$.

\begin{lemma}\label{lem:1a:saves0}
    Let $s \in \RR^m$, $X \in M_{m, n}$ be such that $s + Xa \sim s + XPa$ holds for any $a\in \0^n$ with $0 < a^+ \leq 1$ and any $P \in P(n)$.

    Then $Xa \sim XPa$ holds for any $a\in \0^n$ with $0 < a^+ \leq 1$ and any $P \in P(n)$.
\end{lemma}
\begin{proof}

    If $s = \lambda e$ for some $\lambda \in \RR$, then $s + Xv \sim s +XPv$ is equivalent to $Xv\sim XPv$ for any $v \in \RR^n$ and $X \in M_{m, n}$.

    Thus in the following, we may assume that $s \neq \lambda e$. That is, $\max(s) > \min(s)$. Let $\mc{I} = \{i \in \M \ : \ s_i = \max(s)\}$. Observe that $1 \leq |\mc{I}| < m$.
    
    We prove the lemma by induction on $m$. For $m = 1$ the result is trivial, as it is a particular case of $s = \lambda e$. Assume that the result is proved for all $k < m$.

    Consider arbitrary $a \in \0^n$ with $a^+ \leq 1$ and $P \in P(n)$. 
    Let us show that $Xa\sim XPa$.
    
    Let us choose a sufficiently large $N \in \NN$ such that $N > \max\limits_{i}|(XPa)_i|$ and $N > \max\limits_{i}|(Xa)_i|$. Let us choose $M\in \NN$ such that $M \geq 1$ and $(\max(s) - s_q)M > 2N$ for any $q \in \M \setminus \mc{I}$.
    
    Observe that $\frac{1}{M}a, \frac{1}{M}Pa \in \0^n$ and, since $M\geq 1$, $(\frac{1}{M}a)^+ \leq 1$. Hence $s + \frac{1}{M}Xa \sim s + \frac{1}{M}XPa$.

    Consider arbitrary $i \in \mc{I}$ and $q \in \M \setminus \mc{I}$. Then $$(s+\frac{1}{M}XPa)_i = \max(s) + \frac{1}{M}(XPa)_i \geq \max(s) - \frac{N}{M}.$$

    On the other hand, $$(s+\frac{1}{M}XPa)_q = s_q + \frac{1}{M}(XPa)_q \leq s_q + \frac{N}{M}.$$

    At the same time, $\max(s) - \frac{N}{M} - (s_q + \frac{N}{M}) = (\max(s) - s_q) - \frac{2N}{M} > 0$ by the choice of $M$. It follows that $(s+\frac{1}{M}XPa)_i > (s+\frac{1}{M}XPa)_q$. Similarly, $(s+\frac{1}{M}Xa)_i > (s+\frac{1}{M}Xa)_q$.

    Consider again the equivalence
    
    \begin{equation}\label{eq:sim}
        s + \frac{1}{M}Xa\sim s+\frac{1}{M}XPa
    \end{equation} By proven above, the $|\mc{I}|$ largest entries of the left hand side and the right hand side are located in the same rows, namely, the rows indexed by $\mc{I}$.

    Therefore Equivalence \eqref{eq:sim} remains true when restricted to $\mc{I}$. That is,
    
    \begin{equation}\label{eq:sim_I:with_s}
        s_\mc{I} + (\frac{1}{M}Xa)_{\mc{I}}\sim s_\mc{I} + (\frac{1}{M}XPa)_{\mc{I}}.
    \end{equation}
    
    But $s_\mc{I} = \max(s)e$. As a consequence, Relation \eqref{eq:sim_I:with_s} is equivalent to \begin{equation}\label{eq:_I}
        (Xa)_{\mc{I}}\sim (XPa)_{\mc{I}}
    \end{equation}

    Thus we obtain $s + Xa\sim s + XPa$ and $(s + Xa)_{\mc{I}}\sim (s + XPa)_{\mc{I}}$. It follows that $(s + Xa)_{\M\setminus\mc{I}}\sim (s + XPa)_{\M\setminus\mc{I}}$. Recall that $a \in \0^n$ with $0 < a^+ \leq 1$ and $P\in P(n)$ are arbitrary. 
    
    In other words, the equality $$(s + Xa)_{\M\setminus\mc{I}}\sim (s + XPa)_{\M\setminus\mc{I}}$$ holds for any $a \in \0^n$ with $0 < a^+ \leq 1$ and any $P\in P(n)$. As $1 \leq |\mc{I}|$, we obtain that $|\M\setminus \mc{I}| < m$. Therefore the induction hypothesis is satisfied for $s_{\M\setminus\mc{I}}$ and $X_{\M\setminus \mc{I}}$. Thus $(Xa)_{\M\setminus\mc{I}}\sim (XPa)_{\M\setminus\mc{I}}$. Combining this with Equivalence \eqref{eq:_I} we obtain that $Xa\sim XPa$.
\end{proof}

This allows us to prove the following sufficient condition.
\begin{lemma}\label{lem:s+phi->preserves_0<=1}
    Let $\phi$ be a linear operator on $\RR^n$. Let $s \in \RR^n$. Assume that $s + \phi(a)\sim s + \phi(Pa)$ for any $a \in \0^n$ with $0 < a^+ \leq 1$ and any $P \in P(n)$. Then $\phi$ preserves majorization on $\0^n$.
\end{lemma}
\begin{proof}    
    Let $X$ denote $[\phi]$. Then $s + Xa \sim s + XPa$ for any $a \in \0^n$ with $0 < a^+ \leq 1$ and any $P \in P(n)$.

    Therefore, by Lemma \ref{lem:1a:saves0}, $\phi(a) = Xa \sim XPa = \phi(Pa)$ also holds for any $a \in \0^n$ with $0 < a^+ \leq 1$ and any $P \in P(n)$.
    
    Let us prove that $\phi(b)\sim \phi(Pb)$ for any $b \in \0^n$ and any $P \in P(n)$.
    
     If $b = 0$, then $\phi(b) = \phi(Pb) = 0$ for any $P \in P(n)$.

     Let $b \in \0^n$, $b \neq 0$. Then there is at least one positive entry in $b$. Therefore $b^+ > 0$ and we can consider $a = \frac{1}{b^+} b$. Observe that $a \in \0^n$ and $a^+ = 1$. Then $\phi(a)\sim \phi(Pa)$ for any $P \in P(n)$. On the other hand, $\phi(a)\sim \phi(Pa)$ if and only if $\phi(b)\sim \phi(Pb)$. 
     
     Finally, $\phi(b)\sim \phi(Pb)$ for any $b \in \0^n$ and any $P \in P(n)$. By Lemma \ref{lem:0:prec-sim} this is equivalent to the fact that $\phi$ preserves majorization on $\0^n$.
\end{proof}

We now present a refinement of the previous lemma.

\begin{corollary}\label{cor:s+phi->preserves_0<=1/n}
    Let $\phi$ be a linear operator on $\RR^n$. Let $s \in \RR^n$. Assume that $s + \phi(a)\sim s + \phi(Pa)$ for any $a \in \0^n$ with $0 < a^+ \leq \frac{1}{n}$ and any $P \in P(n)$. Then $\phi$ preserves vector majorization on $\0^n$.
\end{corollary}
\begin{proof}
    Consider arbitrary $b \in \0^n$ with $b^+ \leq 1$. Let $a = \frac{1}{n}b$. Then $a \in \0^n$ and $a^+ \leq \frac{1}{n}$. As a consequence, $s + \phi(a)\sim s + \phi(Pa)$. But then $ns + \phi(b)\sim ns + \phi(Pb)$. As $b \in \0^n$ was arbitrary with $b^+ \leq 1$, we obtain that $\phi$ preserves majorization on $\0^n$ by Lemma \ref{lem:s+phi->preserves_0<=1}.
\end{proof}

Assume that $\Phi$ preserves strong majorization for column stochastic matrices. The next theorem shows that any sum of $\Phi^{j_1}_k, \ldots, \Phi^{j_r}_k$ is a preserver of majorization on $\0^n$. This is the main result of this subsection.

\begin{theorem}\label{thm:Phi^j_k}
    Let a linear operator $\Phi$ on $M_{n, m}$ preserve strong majorization on $\Omega^{col}_{n, m}$. 

    Consider arbitrary $k \in \M$ and arbitrary distinct $j_1, \ldots, j_r \in \M$. 
    
    Then $\sum\limits_{q=1}^r\Phi_k^{j_q}$ preserves majorization on $\0^n$.
\end{theorem}
\begin{proof}
    
    Consider arbitrary $a \in \0^n$ with $a^+ \leq \frac{1}{n}$. Observe that $|a_i| \leq \frac{1}{n}$ for any $i \in \N$. Let $A = \frac{1}{n}J + a\sum\limits_{q=1}^re_{j_q}^t$. Then $e^tA = e^t$ and $A \geq 0$ due to the choice of $a$. That is, $A \in \Omega^{col}_{n, m}$. 
    
    Consider arbitrary $P \in P(n)$. As $A \sim^s PA$, it must be that $\Phi(A)\sim^s \Phi(PA)$. In particular, $\Phi(A)^{(k)}\sim \Phi(PA)^{(k)}$. 

    Let $s = \sum\limits_{q=1}^{m} \Phi^q_k(\frac{1}{n}e)$. Observe that $$\Phi(A)^{(k)} = \sum\limits_{q=1}^{m} \Phi^q_k(A^{(q)}) = \sum\limits_{q=1}^{m} \Phi^q_k(\frac{1}{n}e) + \sum\limits_{q=1}^r\Phi_k^{j_q}(a) = s + (\sum\limits_{q=1}^r\Phi_k^{j_q})(a).$$

    Similarly, $\Phi(PA)^{(k)} = \sum\limits_{q=1}^{m} \Phi^q_k(PA^{(q)}) = \sum\limits_{q=1}^{m} \Phi^q_k(\frac{1}{n}Pe) + \sum\limits_{q=1}^r\Phi_k^{j_q}(Pa) = s + (\sum\limits_{q=1}^r\Phi_k^{j_q})(Pa)$.

    Therefore, we have shown that $s + (\sum\limits_{q=1}^r\Phi_k^{j_q})(a) \sim s + (\sum\limits_{q=1}^r\Phi_k^{j_q})(Pa)$ for any $a \in \0^n$ with $a^+ \leq \frac{1}{n}$ and any $P \in P(n)$.

    Finally, the linear operator $(\sum\limits_{q=1}^r\Phi_k^{j_q})$ preserves vector majorization on $\0^n$ by Corollary \ref{cor:s+phi->preserves_0<=1/n}.
\end{proof}

As a particular case of Theorem \ref{thm:Phi^j_k} for $r = 1$ we obtain the following

\begin{corollary}\label{cor:Phi^j_k}
    Let a linear operator $\Phi$ on $M_{n, m}$ preserve strong majorization on $\Omega^{col}_{n, m}$. Then $\Phi_k^j$ preserves vector majorization on $\0^n$ for any $k, j \in \M$.
\end{corollary}

\subsection{Permutation matrices of $\Phi^j_k$ coincide}

We start by showing that for a fixed $k$ we can take the same permutation matrix in every $\Phi^j_k$. The key idea is to use the fact that due to Theorem \ref{thm:Phi^j_k} and Corollary \ref{cor:Phi^j_k}, the operators $\Phi^i_k$, $\Phi^j_k$ and $\Phi^i_k + \Phi^j_k$ have a similar form. The next lemma shows that this is almost enough to give the required result.

\begin{lemma}\label{lem:sum_of_2_preservers}
    Let $X, Y \in M_n$. Assume that $X = ve^t + \lambda_1 P_1$, $Y = we^t + \lambda_2 P_2$ and $X + Y = se^t + \lambda P$ for some $v, w, s \in \RR^n$, $\lambda_1, \lambda_2, \lambda \in \RR$, $\lambda_1 \neq 0$ and $P_1, P_2, P \in P(n)$. Then one of the following holds.

    \begin{enumerate}
        \item There exist $w' \in \RR^n$ and $\lambda'_2 \in \RR$ such that $Y = w'e^t + \lambda'_2 P_1$;
        \item $n = 3$, $\lambda_1 = \lambda_2$ and $P_1 + P_2 \leq J$ (entry-wise).
    \end{enumerate}
\end{lemma}
\begin{proof}
    If $\lambda_2 = 0$, then $Y = we^t + \lambda_2 P_1$. If $P_1 = P_2$, then there is nothing left to prove. Therefore, assume that $\lambda_2 \neq 0$ and $P_1 \neq P_2$. 
    
    If $n = 2$, then $P_2 = J - P_1$. It follows that $Y=we^t + \lambda_2 (J - P_1) = (w + \lambda_2 e)e^t- \lambda_2 P_1$.

    Now assume that $n \geq 3$. Consider the equation $ve^t + \lambda_1 P_1 + we^t + \lambda_2 P_2 = se^t + \lambda P$. If follows that $(v + w - s)e^t = \lambda P - \lambda_1 P_1 - \lambda_2 P_2$.
    
    In particular, the columns of $\lambda P - \lambda_1 P_1 - \lambda_2 P_2$ are equal. There are two possibilities:
    
    \textbf{Case 1.} $v + w - s = 0$. In this case, $\lambda P = \lambda_1 P_1 + \lambda_2 P_2$. As $\lambda_1, \lambda_2 \neq 0$, it follows that $P_1 = P_2$.

    \textbf{Case 2.} $v + w - s \neq 0$. However, there are at most three nonzero entries in every row of $\lambda P - \lambda_1 P_1 - \lambda_2 P_2$. Thus $n \leq 3$, which means that $n = 3$.

    As $v + w - s \neq 0$, there must be a row without zero entries in $\lambda P - \lambda_1 P_1 - \lambda_2 P_2$. On the other hand, the columns of $\lambda P - \lambda_1 P_1 - \lambda_2 P_2$ are equal. It means that $\lambda = -\lambda_1 = -\lambda_2$.

    Therefore $(v + w - s)e^t = \lambda(P + P_1 + P_2)$. Then $P_1 + P_2 \leq J$. Indeed, otherwise $(P_1 + P_2)_{ij} = 2$ for some $1\leq i, j \leq 3$. This would mean that not all entries of $(P+P_1 + P_2)_{(i)}$ are equal.
\end{proof}

The next lemma shows that for a fixed $k$ we can take the same permutation matrix in every $\Phi^j_k$. 

\begin{lemma}\label{lem:P_in_Phi_k}
     Let a linear operator $\Phi$ on $M_{n, m}$ preserve strong majorization on $\Omega^{col}_{n, m}$. 

    Let $k \in \M$. Then there exist $P \in P(n)$ such that for any $j \in \M$ $$[\Phi^j_k] = v^j e^t + \lambda_j P, \text{ for some } v^j \in \RR^n \text{ and } \lambda_j \in \RR.$$
\end{lemma}
\begin{proof}
    By Corollary \ref{cor:Phi^j_k} every $\Phi^j_k$ preserves vector majorization on $\0^n$. Then for any $j \in \M$ there exist $v^j \in \RR^n$, $\lambda_j \in \RR$ and $P_j \in P(n)$ such that $[\Phi^j_k] = v^j e^t + \lambda_j P_j$. We only need to prove that we can take the same permutation matrix $P_j$ for any $j \in \M$.

    If $\lambda_1 = \ldots = \lambda_m = 0$, then the result is trivial. Assume that $\lambda_i \neq 0$ for some $i \in \M$.

    Consider arbitrary $j \in \M$, $j \neq i$. According to Theorem \ref{thm:Phi^j_k}, the linear operator $\Phi^i_k + \Phi^j_k$ preserves vector majorization on $\0^n$. Then, by Theorem \ref{thm:0}, there exist $s \in \RR^n$, $\lambda \in \RR$ and $P \in P(n)$ such that $[\Phi^i_k + \Phi^j_k] = [\Phi^i_k] + [\Phi^j_k] = se^t + \lambda P$.

    Then, according to Lemma \ref{lem:sum_of_2_preservers}, there are two possibilities. The first is that $[\Phi^j_k] = we^t + \lambda'_j P_i$ for some $w \in \RR^n$ and $\lambda'_j \in \RR$, which is what we need. The second is that $n = 3$, $\lambda_i = \lambda_j$ and $P_i + P_j \leq J$. In order to conclude the proof, we show that the latter is impossible.

    Denote $x = \sum\limits_{q=1}^{m} v^q + (\frac{1}{n}\sum\limits_{r \neq i, j} \lambda_r)e$. If $m = 2$, then the second summand is just $0$.
    
    Consider an arbitrary permutation $l_1, l_2, l_3$ of $1, 2, 3$.
    Define $A \in M_{n, m}$ by $$\begin{cases}
        A^{(i)} = e_{l_1};\\
        A^{(j)} = e_{l_2};\\
        \text{if } m > 2 \text{: } A^{(r)} = \frac{1}{n}e, \text{ for } r \neq i, j.
    \end{cases}$$ Observe that $A \in \Omega^{col}_{n, m}$. 

    Then $$\Phi(A)^{(k)} = \sum\limits_{q = 1}^{m} \Phi^q_k(A^{(q)})= \frac{1}{n}\sum\limits_{r \neq i, j}(v^r e^t + \lambda_r P_r)e + (v^i e^t + \lambda_i P_i)e_{l_1} + (v^j e^t + \lambda_i P_j)e_{l_2}=$$ $$ =\sum\limits_{q=1}^{m}v^q + \frac{1}{n}\sum\limits_{r \neq i, j} \lambda_r P_r e + \lambda_i(P_ie_{l_1} + P_je_{l_2}) = x +\lambda_i(P_ie_{l_1} + P_je_{l_2}).$$

    Similarly, $\Phi(P_{(l_2 l_3)}A)^{(k)} = x + \lambda_i(P_ie_{l_1} + P_j e_{l_3})$. Recall that $l_1, l_2, l_3$ were an arbitrary permutation of $1, 2, 3$. Then as $A \sim^s P_{(l_2 l_3)}A$, we conclude that the equivalence \begin{equation}\label{eq:x+lambda...} x +\lambda_i(P_ie_{l_1} + P_je_{l_2})\sim x +\lambda_i(P_ie_{l_1} + P_je_{l_3})\end{equation} holds for any distinct $l_1, l_2, l_3 \in \{1, 2, 3\}$.

    Let $x' = \frac{1}{\lambda_i}P_i^{-1}x$. Then after multiplying by $\frac{1}{\lambda_i}P_i^{-1}$ Equivalence \eqref{eq:x+lambda...} becomes 

    \begin{equation}\label{eq:x'+...} x' +e_{l_1} + P_i^{-1}P_je_{l_2}\sim x' +e_{l_1} + P_i^{-1}P_je_{l_3}.\end{equation}

    Choose $l_1$ such that $\max(x') = x'_{l_1}$. Observe that $P_i^{-1}P_je_q\neq e_q$ for any $q \in \{1, 2, 3\}$. Indeed, otherwise $P_je_q = P_ie_q$, which contradicts $P_i + P_j \leq J$. Thus $e_{l_1}\neq P_i^{-1}P_je_{l_1}$. Therefore we can choose $l_2$ so that $P_i^{-1}P_je_{l_2} = e_{l_1}$. In addition, it means that $P_i^{-1}P_je_{l_3} = e_{l_2}$.

    Therefore Equation \eqref{eq:x'+...} becomes $$x' + 2e_{l_1}\sim x' + e_{l_1} + e_{l_2}.$$

    But $\max(x' + 2e_{l_1}) = \max(x') + 2$, while $\max(x' + e_{l_1} + e_{l_2}) = \max(x') + 1$, a contradiction.

    Finally, we have shown that only the first option provided by Lemma \ref{lem:sum_of_2_preservers} is possible, which is what was required.
    
\end{proof}

As it was shown in Lemma \ref{lem:1:sum_preserves}, $\sum\limits_{j=1}^m \Phi^j_k$ preserves vector majorization. According to Theorem \ref{Ando}, there are two types of such operators. In Lemma \ref{lem:Phi-P=>sum...} we show that if some $\Phi^j_k$ has a nontrivial permutation component, we can always assume the second type of Ando's operators with the same permutation matrix. First, we make the following observation.

\begin{lemma}\label{lem:s+e_g-e_h}
    Let $s \in \RR^n$. If $s + e_g - e_h\sim s + P(e_g - e_h)$ holds for any $P \in P(n)$ and any $g, h \in \N$, then $s = \lambda e$ for some $\lambda \in \RR$. 
\end{lemma}
\begin{proof}
    Assume that there is no $\lambda \in \RR$ such that $s = \lambda e$. Then $\max(s) > \min(s)$.

    Let us choose $g$ such that $s_g = \max(s)$. Let us choose $P \in P(n)$ such that $Pe_g = e_w$, where $s_w = \min(s)$.

    If follows that $\max(s + e_g - e_h) = \max(s) + 1$, while $\max(s + P(e_g - e_h)) < \max(s) + 1$, a contradiction.
\end{proof}

\begin{lemma}\label{lem:Phi-P=>sum...}
    Let a linear operator $\Phi$ on $M_{n, m}$ preserve strong majorization on $\Omega^{col}_{n, m}$.

    Assume that for some $j_1, k \in \M$ $$[\Phi^{j_1}_k] = ve^t + \gamma P \text{ for some } v \in \RR^n, \gamma \in \RR \text{ and } P \in P(n).$$

    If $\gamma \neq 0$, then $$[\sum\limits_{j=1}^{m}\Phi^j_k] = \alpha J + \beta P \text{ for some } \alpha, \beta \in \RR.$$
\end{lemma}
\begin{proof}
    By Lemma \ref{lem:1:sum_preserves} $\sum\limits_{j=1}^m \Phi^j_k$ preserves vector majorization. According to Theorem \ref{Ando} there are two possibilities
    
    \textbf{Case 1.} $[\sum\limits_{j=1}^{m}\Phi^j_k] = se^t$ for some $s \in \RR^n$. We will show that $s = \lambda e$ for some $\lambda \in \RR$.

    Consider $A = \frac{1}{n}J + \frac{1}{n}(e_g - e_h)e_{j_1}^t$, where $g, h \in \N$ are arbitrary distinct indices. Observe that $A\in \Omega_{n, m}^{col}$.

    Then $\Phi(A)^{k} = \sum\limits_{j=1}^{m}\Phi^j_k (\frac{1}{n}e) + \Phi^{j_1}_k(\frac{1}{n}(e_g - e_h)) = s + \frac{\gamma}{n}P(e_g - e_h)$.

    Similarly, for any $R \in P(n)$ we have $\Phi(RA)^{(k)} = s + \frac{\gamma}{n}PR(e_g - e_h)$.

    But as $\Phi$ preserves strong majorization on $\Omega_{n, m}^{col}$, we obtain that $$s + \frac{\gamma}{n}P(e_g - e_h) \sim s + \frac{\gamma}{n}PR(e_g - e_h).$$

    Multiplying by $\frac{n}{\gamma}P^{-1}$, we get $$\frac{n}{\gamma}P^{-1}s + (e_g - e_h)\sim \frac{n}{\gamma}P^{-1}s + R(e_g - e_h).$$

    By Lemma \ref{lem:s+e_g-e_h}, $\frac{n}{\gamma}P^{-1}s = \lambda e$ for some $\lambda \in \RR$. It follows that $s = \frac{\gamma \lambda}{n}e$ and $$[\sum\limits_{j=1}^{m}\Phi^j_k] = \frac{\gamma \lambda}{n} J + 0 P.$$

    \textbf{Case 2.} $[\sum\limits_{j=1}^{m}\Phi^j_k] = \alpha J + \beta Q$ for some $\alpha, \beta \in \RR$ and $Q \in P(n)$. 
    
    If $\beta = 0$, or $Q = P$, then there is nothing to prove. Assume that $\beta \neq 0$ and $Q \neq P$. If $n = 2$, then $Q = J - P$ and $[\sum\limits_{j=1}^{m}\Phi^j_k] = (\alpha + \beta) J - \beta P$.

    Assume that $n \geq 3$. Then by Lemma \ref{lem:P_in_Phi_k} there exist $P' \in P(n)$, $v^1, \ldots, v^m \in \RR^n$ and $\lambda_1, \ldots, \lambda_m \in \RR$ such that $$[\Phi^j_k] = v^j e^t + \lambda_j P', \text{ for any } j \in \M.$$

    Observe that $[\Phi^{j_1}_k] = ve^t + \gamma P = v^{j_1}e^t + \lambda_{j_1}P'$.

    It follows that $(v-v^{j_1})e^t = \lambda_{j_1} P' - \gamma P$. Recall that $\gamma \neq 0$. The columns of the left hand side are equal. Hence, the columns of the right hand side are equal, which given $n > 2$ means that $P' = P$.

    Therefore, $[\Phi^j_k] = v^j e^t + \lambda_j P, \text{ for any } j \in \M.$ It  follows that $\alpha J + \beta Q = (\sum\limits_{j=1}^m v^j) e^t + (\sum\limits_{j=1}^m \lambda_j)P$. That is, $$(\alpha e - \sum\limits_{j=1}^m v^j) e^t = (\sum\limits_{j=1}^m \lambda_j)P - \beta Q.$$

    Here $\beta \neq 0$ and $n > 2$ and as before, we come to the conclusion that $Q = P$.
\end{proof}

We have shown that for a fixed $k$ the permutation matrices of $\Phi^j_k$ coincide. It remains to show that the same permutation matrix can be used for all $k$. We start with the following observation that will allow us to do so.

\begin{lemma}\label{lem:e_g-e_h}
    Let $P \in P(n)$, $n \geq 3$, be such that \begin{equation}\label{eq:e_g-e_h|Pe_g-e_h}\left( e_g - e_h \ | \ P(e_g - e_h)\right) \sim^s \left( e_g - e_h \ | \ Q^{-1}PQ(e_g - e_h) \right) \text{ for any } g, h \in \N \text{ and } Q \in P(n).\end{equation}

    Then $P = I$.
\end{lemma}
\begin{proof}
    Assume that $P \neq I$. Then there exist distinct $q, r \in \N$ such that $Pe_q = e_r$. Consider Equivalence \eqref{eq:e_g-e_h|Pe_g-e_h} for $g = q$, $h = r$ and $Q = P_{(rs)}$, where $s \neq q, r$. Note that $P_{(rs)}PP_{(rs)}e_q = P_{(rs)}Pe_q=e_s$ and $P_{(rs)}PP_{(rs)}e_r = P_{(rs)}Pe_s$.

    As a consequence, we obtain $$\left( e_q - e_r \ | \ e_r - Pe_r \right) \sim^s \left( e_q - e_r \ | \ e_s - P_{(rs)}Pe_s \right).$$

    Observe that the left-hand side matrix contains a row $(-1 \ | \ 1)$, while the right-hand side does not, a contradiction.
\end{proof} 

Finally, we show that we can assume that the permutation matrices of all $\Phi^j_k$ coincide, which is the goal of this subsection.

\begin{lemma}\label{lem:all_P}
    Let a linear operator $\Phi$ on $M_{n, m}$ preserve strong majorization on $\Omega^{col}_{n, m}$.

    Then there exists $P \in P(n)$ such that for every $i, j \in \M$
    
    $$[\Phi^j_i] = v^j_i e^t + \lambda^j_i P \text{ for some } v^j_i \in \RR^n,\text{ and } \lambda^j_i \in \RR .$$
\end{lemma}
\begin{proof}
    By Lemma \ref{lem:P_in_Phi_k} for any $i \in \M$ there exists $P_i \in P(n)$ such that $$[\Phi^j_i] = v^j_i e^t + \lambda^j_i P_i, \text{ for some } v^j_i \in \RR^n \text{ and } \lambda^j_i \in \RR.$$

    We only need to prove that we can always take $P_1 = \ldots = P_m$. Once again, in case $n = 2$ we can freely choose permutation matrices. Assume that $n \geq 3$.

    If $\lambda^j_i = 0$ for all $j \in \M$, then we can freely choose any permutation matrices in operators $\Phi^j_i$.

    Assume that $P_{i_1}\neq P_{i_2}$, while $\lambda^{j_1}_{i_1} \neq 0$ for some $j_1 \in \M$ and $\lambda^{j_2}_{i_2} \neq 0$ for some $j_2 \in \M$.
    There are two possibilities.

    \textbf{Case 1.} There exists $j_1 \in \M$ such that $\lambda^{j_1}_{i_1}, \lambda^{j_1}_{i_2} \neq 0$.

    Consider arbitrary $g, h \in \M$ and $Q \in P(n)$. Let $A = \frac{1}{n}((e_g - e_h)e^t_{j_1} + J)$. Observe that $A \in \Omega^{col}_{n, m}$. In addition, $QA = \frac{1}{n}(Q(e_g - e_h)e^t_{j_1} + J)$.

    As $A \sim^s QA$, we obtain $\Phi(A)\sim^s \Phi(QA)$. In particular, $(\Phi(A))^{(i_1, i_2)}\sim^s (\Phi(QA))^{(i_1, i_2)}$.

    Due to Lemma \ref{lem:Phi-P=>sum...}, we have $\sum\limits_{j=1}^{m}[\Phi^j_{i_1}] = \alpha_{i_1} J + \beta_{i_1} P_{i_1}$ and $\sum\limits_{j=1}^{m}[\Phi^j_{i_2}] = \alpha_{i_2} J + \beta_{i_2} P_{i_2}$ for some $\alpha_{i_1}, \alpha_{i_2}, \beta_{i_1}, \beta_{i_2} \in \RR$.

    Therefore $$\begin{aligned}
        \Phi(A)^{(i_1)} &= \sum\limits_{j=1}^{m} \Phi^j_{i_1}(A^{(j)}) = \frac{1}{n}(\Phi^{j_1}_{i_1}(e_g - e_h)+\sum\limits_{j=1}^{m} \Phi^j_{i_1}(e)) = \frac{1}{n}(v^{j_1}_{i_1}e^t(e_g - e_h) + \lambda^{j_1}_{i_1}P_{i_1}(e_g - e_h)+\alpha_{i_1}ne + \beta_{i_1}e) = \\ &= (\alpha_{i_1} + \frac{1}{n}\beta_{i_1})e +\frac{1}{n}\lambda^{j_1}_{i_1}P_{i_1}(e_g - e_h).
    \end{aligned}$$

Similarly, 

    $$\begin{aligned}
    \Phi(A)^{(i_2)} &= (\alpha_{i_2} + \frac{1}{n}\beta_{i_2})e +\frac{1}{n}\lambda^{j_1}_{i_2}P_{i_2}(e_g - e_h);\\
    \Phi(QA)^{(i_1)} &= (\alpha_{i_1} + \frac{1}{n}\beta_{i_1})e +\frac{1}{n}\lambda^{j_1}_{i_1}P_{i_1}Q(e_g - e_h);\\
    \Phi(QA)^{(i_2)} &= (\alpha_{i_2} + \frac{1}{n}\beta_{i_2})e +\frac{1}{n}\lambda^{j_1}_{i_2}P_{i_2}Q(e_g - e_h).
    \end{aligned}$$

    Due to Lemma \ref{lem:ev^t}, we can dispense with $(\alpha_{i_1} + \frac{1}{n}\beta_{i_1})e$ and $(\alpha_{i_2} + \frac{1}{n}\beta_{i_2})e$ in $(\Phi(A))^{(i_1, i_2)}\sim^s (\Phi(QA))^{(i_1, i_2)}$.

    As a consequence, we obtain $$\left(\begin{array}{c|c}
        \frac{1}{n}\lambda^{j_1}_{i_1}P_{i_1}(e_g - e_h) & \frac{1}{n}\lambda^{j_1}_{i_2}P_{i_2}(e_g - e_h) 
    \end{array}\right)\sim^s \left(\begin{array}{c|c}
        \frac{1}{n}\lambda^{j_1}_{i_1}P_{i_1}Q(e_g - e_h) & \frac{1}{n}\lambda^{j_1}_{i_2}P_{i_2}Q(e_g - e_h) 
    \end{array}\right).$$

    Recall that $\lambda^{j_1}_{i_1}, \lambda^{j_1}_{i_2} \neq 0$. Then, due to Corollary \ref{cor:AD<BD} we can get rid of the coefficients to obtain $$\left(\begin{array}{c|c}
        P_{i_1}(e_g - e_h) & P_{i_2}(e_g - e_h) 
    \end{array}\right)\sim^s \left(\begin{array}{c|c}
        P_{i_1}Q(e_g - e_h) & P_{i_2}Q(e_g - e_h)
    \end{array}\right).$$

    Let us multiply both sides by $P_{i_1}^{-1}$ and after that multiply the right-hand side by $Q^{-1}$. Thus we obtain $$\left(\begin{array}{c|c}
        e_g - e_h & P_{i_1}^{-1}P_{i_2}(e_g - e_h) 
    \end{array}\right)\sim^s \left(\begin{array}{c|c}
       e_g - e_h & Q^{-1}P_{i_1}^{-1}P_{i_2}Q(e_g - e_h)
    \end{array}\right).$$

    Note that this equivalence holds for any $g, h \in \M$ and any $Q \in P(n)$. Then, by Lemma \ref{lem:e_g-e_h}, $P_{i_1}^{-1}P_{i_2} = I$, which contradicts $P_{i_1}\neq P_{i_2}$.
    
\textbf{Case 2.} There does not exist $j \in \M$ such that $\lambda^{j}_{i_1}, \lambda^{j}_{i_2} \neq 0$. However, there exist $j_1, j_2 \in \M$ such that $\lambda^{j_1}_{i_1}, \lambda^{j_2}_{i_2} \neq 0$. In this case, $\lambda^{j_2}_{i_1} = \lambda^{j_1}_{i_2} = 0$. This case is very similar to the previous one. We just take $A =\frac{1}{n}((e_g - e_h)(e^t_{j_1} + e^t_{j_2}) + J)$ to get the same relations as in Case 1.

For example, $$\begin{aligned}
        \Phi(A)^{(i_1)} &= \sum\limits_{j=1}^{m} \Phi^j_{i_1}(A^{(j)}) = \frac{1}{n}(\Phi^{j_1}_{i_1}(e_g - e_h) + \Phi^{j_2}_{i_1}(e_g - e_h) + \sum\limits_{j=1}^{m} \Phi^j_{i_1}(e)) \\ &= \frac{1}{n}(v^{j_1}_{i_1}e^t(e_g - e_h) + \lambda^{j_1}_{i_1}P_{i_1}(e_g - e_h) + v^{j_2}_{i_1}e^t(e_g - e_h) + \lambda^{j_2}_{i_1}P_{i_1}(e_g - e_h)+\alpha_{i_1}ne + \beta_{i_1}e) = \\ &= (\alpha_{i_1} + \frac{1}{n}\beta_{i_1})e +\frac{1}{n}\lambda^{j_1}_{i_1}P_{i_1}(e_g - e_h).
    \end{aligned}$$

    Similarly, 

    $$\begin{aligned}
    \Phi(A)^{(i_2)} &= (\alpha_{i_2} + \frac{1}{n}\beta_{i_2})e +\frac{1}{n}\lambda^{j_2}_{i_2}P_{i_2}(e_g - e_h);\\
    \Phi(QA)^{(i_1)} &= (\alpha_{i_1} + \frac{1}{n}\beta_{i_1})e +\frac{1}{n}\lambda^{j_1}_{i_1}P_{i_1}Q(e_g - e_h);\\
    \Phi(QA)^{(i_2)} &= (\alpha_{i_2} + \frac{1}{n}\beta_{i_2})e +\frac{1}{n}\lambda^{j_2}_{i_2}P_{i_2}Q(e_g - e_h).
    \end{aligned}$$

    The rest is identical to the case above and leads to $P_{i_1} = P_{i_2}$.
\end{proof}

\subsection{Characterization}

Lemma \ref{lem:all_P} already leads to a strong necessary condition.

\begin{theorem}\label{thm:necessary_form}
    Let a linear operator $\Phi$ on $M_{n, m}$ preserve strong majorization on $\Omega^{col}_{n, m}$. Then $$\Phi(X) = \sum\limits_{j=1}^{m} (e^t X^{(j)})S_j + PXR \text{ for some } S_1, \ldots, S_m \in M_{n, m}, \ R \in M_m \text{ and } P \in P(n).$$
\end{theorem}
\begin{proof}
    By Lemma \ref{lem:all_P}, there exists $P \in P(n)$ such that for every $i, j \in \M$
    
    $$[\Phi^j_i] = v^j_i e^t + \lambda^j_i P \text{ for some } v^j_i \in \RR^n,\text{ and } \lambda^j_i \in \RR .$$

    For every $j \in \M$ let $S_j = \left(\begin{array}{c|c|c}
        v^j_1 & \ldots & v^j_m
    \end{array}\right) \in M_{n, m}.$

    Define $R \in M_{m}$ by $r_{ij} = \lambda^i_j$.

    Then $$\begin{aligned}
        \Phi(X) &= \left(\begin{array}{c|c|c}
            \sum\limits_{j=1}^m \Phi^j_1(X^{(j)})  \ & \  \ldots  \ & \  \sum\limits_{j=1}^m \Phi^j_m(X^{(j)})
        \end{array}\right)=\left(\begin{array}{c|c|c}
            \sum\limits_{j=1}^m v^j_1 e^tX^{(j)} + \sum\limits_{j=1}^m \lambda^j_1 P X^{(j)}  \ & \   \ldots  \ & \   \sum\limits_{j=1}^m v^j_m e^tX^{(j)} + \sum\limits_{j=1}^m \lambda^j_m P X^{(j)}
        \end{array}\right) \\ 
        &= \left(\begin{array}{c|c|c}
            \sum\limits_{j=1}^m (e^tX^{(j)})v^j_1  \ & \   \ldots  \ & \  \sum\limits_{j=1}^m (e^tX^{(j)})v^j_m
        \end{array}\right) + \left(\begin{array}{c|c|c}
            \sum\limits_{j=1}^m \lambda^j_1 P X^{(j)}  \ & \   \ldots  \ & \   \sum\limits_{j=1}^m \lambda^j_m P X^{(j)}
        \end{array}\right) \\ 
        &=\sum\limits_{j=1}^m (e^tX^{(j)})\left(\begin{array}{c|c|c}
            v^j_1  \ & \  \ldots  \ & \   v^j_m
        \end{array}\right) + P\left(\begin{array}{c|c|c}
            \sum\limits_{j=1}^m \lambda^j_1 X^{(j)}  \ & \   \ldots  \ & \   \sum\limits_{j=1}^m \lambda^j_m X^{(j)}
        \end{array}\right)= \sum\limits_{j=1}^{m} (e^t X^{(j)})S_j + PXR.
    \end{aligned}$$
\end{proof}

\begin{remark}\label{rem:Phi(A)}
Let $\Phi(X) = \sum\limits_{j=1}^{m} (e^t X^{(j)})S_j + PXR \text{ for some } S_1, \ldots, S_m \in M_{n, m}, \ R \in M_m \text{ and } P \in P(n)$. 

If $A \in \Omega^{col}_{n, m}$, then $\Phi(A) = \sum\limits_{j=1}^{m} S_j + PAR$.
\end{remark}

 As we shall see next, the necessary condition provided by Theorem \ref{thm:necessary_form} is not sufficient.

\begin{corollary}\label{cor:necessary}
    Let a linear operator $\Phi$ on $M_{n, m}$ preserves strong majorization on $\Omega_{n, m}$ and $$\Phi(X) = \sum\limits_{j=1}^{m} (e^t X^{(j)})S_j + PXR \text{ for some } S_1, \ldots, S_m \in M_{n, m}, \ R \in M_m \text{ and } P \in P(n).$$

    If $R \neq 0$, then $\sum\limits_{j=1}^m S_j = e v^t$ for some $v \in \RR^m$.
\end{corollary}
\begin{proof}

    Let $S = \sum\limits_{j=1}^m S_j$. If $R = 0$, then there is nothing to prove. Consider arbitrary $l \in \M$ with $R^{(l)}\neq 0$. Let $k \in \M$ such that $r_{kl} \neq 0$. 
    For arbitrary $g, h \in \N$ and $Q \in P(n)$ define $A = \frac{1}{n}(P^{-1}(e_g - e_h)e_k^t + J)$ and $B = \frac{1}{n}(P^{-1}Q(e_g - e_h)e_k^t + J)$. As $A, B \in \Omega^{col}_{n, m}$ and $A\sim^s B$, we conclude that $\Phi(A)\sim^s \Phi(B)$.

    Due to Remark \ref{rem:Phi(A)}, $\Phi(A) = S + PAR=S + \frac{1}{n}((e_g - e_h)R_{(k)} + JR)$. Dually,  $\Phi(B) =S + \frac{1}{n}(Q(e_g - e_h)R_{(k)} + JR)$. By Corollary \ref{cor:+J} we obtain 
    
    \begin{equation}\label{eq:cor:necessary}
        S + (e_g - e_h)R_{(k)} \sim^s S + Q(e_g - e_h)R_{(k)}
    \end{equation}

    Considering the $l$-th column of \eqref{eq:cor:necessary} we obtain $S^{(l)} + r_{kl}(e_g - e_h) \sim^s S^{(l)} + r_{kl}Q(e_g - e_h)$. As $g, h, Q$ are arbitrary and $r_{kl}\neq 0$, we can apply Lemma \ref{lem:e_g-e_h} and conclude that $S^{(l)} = v_l e$ for some $v_l \in \RR$.
    
    It remains to show that $S^{(j)} = v_j e$ for any $j\in \M$ with $R^{(j)} = 0$. Multiplying \eqref{eq:cor:necessary} by $e_l +e_j$ (multiplication on the right by the same matrix preserves majorization) we obtain $S^{(l)} + S^{(j)} + r_{kl}(e_g - e_h) \sim^s S^{(l)} + S^{(j)} + r_{kl}Q(e_g - e_h)$. As $S^{(l)}=v_l e$, the latter is equivalent to $S^{(j)} + r_{kl}(e_g - e_h) \sim^s S^{(j)} + r_{kl}Q(e_g - e_h)$. Same as before, this means that $S^{(j)} = v_j e$ for some $v_j \in \RR$.

    Finally, $R \neq 0$ implies $S = e (v_1 \ \ldots \ v_m)$.
\end{proof}

The following is the main result of the paper.

\begin{theorem}\label{thm}
    A linear operator $\Phi$ on $M_{n, m}$ preserves strong majorization on $\Omega^{col}_{n, m}$ if and only if the following holds:
    
    $$\Phi(X) = \sum\limits_{j=1}^{m} (e^t X^{(j)})S_j + PXR \text{ for some } S_1, \ldots, S_m \in M_{n, m}, \ R \in M_m \text{ and } P \in P(n).$$ Moreover, if $R \neq 0$, then $\sum\limits_{j=1}^m S_j = e v^t$ for some $v \in \RR^m$.
\end{theorem}
\begin{proof}
    The necessity was proven in Corollary \ref{cor:necessary}. It just remains to prove that the linear operator of this form does preserve strong majorization on $\Omega^{col}_{n, m}$. If $R = 0$, then the result follows from Theorem \ref{thm:LiPoon}. Assume that $R\neq 0$. Then $\sum\limits_{j=1}^m S_j = e v^t$ for some $v \in \RR^m$.
    
    Consider arbitrary $A, B \in \Omega^{col}_{n,m}$. Then $\Phi(A) = ev^t + PAR$ and $\Phi(B) = ev^t + PBR$.
    
    Assume that $A \preceq^s B$. Then $PAR \preceq PBR$ by Theorem \ref{thm:LiPoon}. By Lemma \ref{lem:ev^t} it follows that $\Phi(A)\preceq^s \Phi(B)$.
\end{proof}

Note that the form of operators in Theorem \ref{thm:necessary_form} is a very natural structure that encompasses both types of Li Poon operators in Theorem \ref{thm:LiPoon}. The same is true of the main result: the two forms of operators in Theorem \ref{thm:LiPoon} are just particular cases of operators preserving majorization on $\Omega^{col}_{n, m}$ from Theorem \ref{thm}.

\begin{remark}
    Naturally, Theorem \ref{thm:vector} is a particular case of Theorem \ref{thm} for $m = 1$.
\end{remark}

In Section \ref{sec:Reduction} we have shown that majorization can always be reduced to the column stochastic matrices. Despite that, restriction to column stochastic matrices provides a richer structure of linear preservers, as follows from Theorem \ref{thm}. The following example illustrates that.

\begin{example}\label{ex:last}
Consider a linear operator $\Phi(X)$ on $M_2$ defined by $\Phi(X) = (e^t X^{(1)} - e^t X^{(2)})I + X$.

Then $\Phi(C) = C$ for any $C \in \Omega^{col}_2$. 

On the other hand, consider, for example $A = \begin{pmatrix}
    1 & 0\\0 & 0
\end{pmatrix}$ and $B = \begin{pmatrix}
    0 & 0\\1 & 0
\end{pmatrix}$. Clearly, $A \sim^s B$. However, $\Phi(A) = \begin{pmatrix}
    2 & 0\\0 & 1
\end{pmatrix}$ and $\Phi(B) = \begin{pmatrix}
    1 & 0 \\1 & 1
\end{pmatrix}$. It follows that $\Phi$ does not preserve strong majorization, since $\Phi(A) \not\preceq^s \Phi(B)$.

This example also shows that the components of a linear preserver of strong majorization for column stochastic matrices (see the remarks after Lemma \ref{lem:Phi:decomposition}) do not, in general, preserve majorization for probability distributions.

Take, for example, $\Phi^1_1(v) = \Phi(ve_1^t)e_1 = ((e^t v) I + ve_1^t)e_1 = \begin{pmatrix}
    2v_1 + v_2 \\ v_2
\end{pmatrix}$. This operator does not preserve majorization for probability distributions, since $\Phi^1_1(\begin{pmatrix}
    1 \\ 0
\end{pmatrix}) = \begin{pmatrix}
    2 \\ 0
\end{pmatrix} \not\preceq \Phi^1_1(\begin{pmatrix}
    0 \\ 1
\end{pmatrix}) = \begin{pmatrix}
    1 \\ 1
\end{pmatrix}$.
\end{example}

\section*{Acknowledgments} The author thanks Professor Alexander Guterman for valuable discussions on majorization theory and linear preserver problems.

{\small
\bibliographystyle{plain}
\bibliography{mybibfile}

\begin{thebibliography}{10}

\bibitem{Ando89}
T.~Ando.
\newblock Majorization, doubly stochastic matrices, and comparison of
  eigenvalues.
\newblock {\em Linear Algebra Appl.}, 118:163--248, 1989.

\bibitem{BeasleyLee}
L.B. Beasley and S.-G. Lee.
\newblock Linear operators preserving multivariate majorization.
\newblock {\em Linear Algebra Appl.}, 304 1(1):141--159, 2000.

\bibitem{BeasleyLeeLee}
L.B. Beasley, S.-G. Lee, and Y.-H. Lee.
\newblock A characterization of strong preservers of matrix majorization.
\newblock {\em Linear Algebra Appl.}, 367:341--346, 2003.

\bibitem{Dahl99b}
G.~Dahl.
\newblock Majorization polytopes.
\newblock {\em Linear Algebra Appl.}, 297:157 -- 175, 1999.

\bibitem{Dahl99}
G.~Dahl.
\newblock Matrix majorization.
\newblock {\em Linear Algebra Appl.}, 288:53 -- 73, 1999.

\bibitem{DGS1_MMC}
G.~Dahl, A.~Guterman, and P.~Shteyner.
\newblock Majorization for matrix classes.
\newblock {\em Linear Algebra Appl.}, 555:201 -- 221, 2018.

\bibitem{DGS2_01}
G.~Dahl, A.~Guterman, and P.~Shteyner.
\newblock Majorization for (0,1)-matrices.
\newblock {\em Linear Algebra Appl.}, 585:147 -- 163, 2020.

\bibitem{DGS3_0+-1}
G.~Dahl, A.~Guterman, and P.~Shteyner.
\newblock Majorization for $(0, \pm 1)$-matrices.
\newblock {\em Linear Algebra Appl.}, 663:200--221, 2023.

\bibitem{Gour}
G.~Gour, D.~Jennings, F.~Buscemi, R.~Duan, and I.~Marvian.
\newblock Quantum majorization and a complete set of entropic conditions for
  quantum thermodynamics.
\newblock {\em Nature Communications}, 9(1), 2018.

\bibitem{GS_tuples}
A.~Guterman and P.~Shteyner.
\newblock Linear operators preserving majorization of matrix tuples.
\newblock {\em Vestnik St. Petersburg Univ. Math.}, 53:136–144, 2020.

\bibitem{GS2_convecters}
A.~Guterman and P.~Shteyner.
\newblock Linear converters of weak, directional and strong majorizations.
\newblock {\em Linear Algebra Appl.}, 613:320--346, 2021.

\bibitem{GS3_preservers_01}
A.~Guterman and P.~Shteyner.
\newblock Linear operators preserving strong majorization of (0,1)-matrices.
\newblock {\em Linear Algebra Appl.}, 658:116--150, 2023.

\bibitem{HR}
A.M. Hasani and M.~Radjabalipour.
\newblock Linear preserver of matrix majorization.
\newblock {\em Int. J. Pure Appl. Math.}, 32 4:475--482, 2006.

\bibitem{HR2}
A.M. Hasani and M.~Radjabalipour.
\newblock On linear preservers of (right) matrix majorization.
\newblock {\em Linear Algebra Appl.}, 423(2-3):255--261, June 2007.

\bibitem{LiPierce}
C.-K. Li and S.~Pierce.
\newblock Linear preserver problems.
\newblock {\em Amer. Math. Monthly}, 108:591 -- 605, 2001.

\bibitem{LiPoon2001}
C.-K. Li and E.~Poon.
\newblock Linear operators preserving directional majorization.
\newblock {\em Linear Algebra Appl.}, 325(1):141 -- 146, 2001.

\bibitem{MaOlAr11}
A.~W. Marshall, I.~Olkin, and B.~C. Arnold.
\newblock {\em Inequalities: Theory of Majorization and Its Applications}.
\newblock Springer, New York, second edition, 2011.

\bibitem{Nielsen}
M.A. Nielsen.
\newblock An introduction of majorization and its applications to quantum
  mechanics.
\newblock Technical report, Department of Physics, University of Queensland,
  Australia, 2002.
\newblock Available at
  \url{http://michaelnielsen.org/blog/talks/2002/maj/book.ps}.

\bibitem{PeMaSi05}
F.~D.~Martínez Pería, F.~G. Massey, and L.~E. Silvestre.
\newblock Weak matrix majorization.
\newblock {\em Linear Algebra Appl.}, 403:343 -- 368, 2005.

\bibitem{Pierce}
S.~Pierce et~al.
\newblock A survey of linear preserver problems.
\newblock {\em Linear Multilinear Algebra}, 33 1-2, 1992.

\bibitem{Schur}
I.~Schur.
\newblock Uber eine klasse von mittelbildungen mit anwendungen auf die
  determinantentheorie.
\newblock {\em Sitzungsberichte der Berliner Mathematischen Gesellschaft},
  22(9-20):51, 1923.

\bibitem{Shteyner_CCM}
P.~M. Shteyner.
\newblock Converting column majorization.
\newblock {\em J. Math. Sci. (N.Y.)}, 255(3):340--352, April 2021.

\bibitem{Shteyner_conv_01}
P.~M. Shteyner.
\newblock Linear operators preserving and converting majorizations of the (0,
  1)-vectors.
\newblock {\em J. Math. Sci. (N.Y.)}, 272(4):615--624, 5 2023.

\bibitem{Shteyner_column:01}
P.~M. Shteyner.
\newblock Linear operators preserving column majorization of the (0,
  1)-vectors.
\newblock {\em J. Math. Sci. (N.Y.)}, 281(2):312--333, 4 2024.

\bibitem{Torgersen}
E.~Torgersen.
\newblock Stochastic orders and comparison of experiments.
\newblock {\em Lecture Notes-Monograph Series}, pages 334--371, 1991.

\end{thebibliography}
}
\end{document}